\documentclass{amsart}
\usepackage{amsmath}
\usepackage{amsfonts}
\usepackage{amssymb}
\newtheorem{theorem}{Theorem}[section]
\newtheorem{theo}{Theorem}
\newtheorem{lm}[theorem]{Lemma}
\newtheorem{cor}[theorem]{Corollary}
\newtheorem{df}[theo]{Definition}
\newtheorem{rem}[theorem]{Remark}
\newtheorem{pr}[theorem]{Proposition}

\begin{document}

\title{Supersymmetric elements in divided powers algebras}
\author{Franti\v sek Marko}
\address{Pennsylvania State University \\ 76 University Drive \\ Hazleton, PA 18202 \\ USA}
\email{fxm13@psu.edu}
\begin{abstract}
Description of adjoint invariants of general Linear Lie superalgebras $\mathfrak{gl}(m|n)$ by Kantor and Trishin is given in terms of supersymmetric polynomials. 
Later, generators of invariants of the adjoint action of the general linear supergroup $GL(m|n)$ and generators of supersymmetric polynomials were determined over fields of positive characteristic. 
In this paper, we introduce the concept of supersymmetric elements in the divided powers algebra $Div[x_1, \ldots, x_m,y_1, \ldots, y_n]$, and give a characterization of supersymmetric elements via a system of linear equations. Then we determine generators of supersymmetric elements for divided powers algebras in the cases when $n=0$, $n=1$, and 
$m\leq 2, n=2$.   
\end{abstract}
\maketitle

\section*{Introduction and notation}

We start by recalling a description of invariants of conjugacy classes of matrices which is a classical problem in the invariant theory - see Chapter 1 of \cite{springer}. 
Let $K$ be an infinite field of characteristic zero, $V$ be a $K$-space of dimension $m$, $E$ be the space of all $K$-linear maps of $V$ 
(a choice of a basis of $V$ identifies $E$ with $m\times m$-matrices). The general linear group $GL(V)$
acts on $E$ via conjugation $g.a=gag^{-1}$ for $g\in GL(V)$ and $a\in E$ (corresponding to a change of basis of $V$).

Consider the space $S(E)$ of all $K$-valued polynomial functions on $E$ together with the action of $GL(V)$ given as $g.f(a)=f(g^{-1}a)$ for $g\in GL(V)$, $f\in S(E)$ and $a\in E$.
By Chevalley's restriction theorem (see Theorem 1.5.7 of \cite{springer}), the ring of invariants $S(E)^{GL(V)}$ is isomorphic to the ring of symmetric polynomials $K[x_1, \ldots, x_m]$.
Here the variables $x_1, \ldots, x_m$ are related to the coefficients $\sigma_i(M)$ of the characteristic polynomial of $m\times m$-matrices $M$.

The vertices of the Dynkin diagram $\Delta=A_m$ corresponding to $GL(m)$ are given by simple positive roots that can be labeled by $1, \ldots, m$. The action of the Weyl group $W$ permutes the elements of the set $\{1, \ldots, m\}$. If we denote by $P$ the ring of polynomials in variables $x_1, \ldots, x_m$, then the ring of symmetric polynomials
$P^W$ consists of invariants of $P$ under the action induced by the Weyl group $W$.

The above classical correspondence was extended to the case of general linear superalgebras $GL(m|n)$ in characteristic zero by Kantor and Trishin in \cite{kanttrish}. They have described the 
polynomial invariants of general linear Lie supergroup $\mathfrak{gl}(m|n)$ and their connection to the algebra $A_s$ of supersymmetric polynomials. 
The algebra $A_s$ consists of polynomials $f (x|y) = f (x_1, \ldots , x_m, y_1, \ldots , y_n)$ that are symmetric in variables
$x_1, \ldots, x_m$ and $y_1, \ldots, y_n$ separately, such that $\frac{d}{dT} f (x|y)|_{x_1=y_1=T}=0$.
Before the paper of Kantor and Trishin, motivated by a work of Kac and Schneuert, Stembridge proved a conjecture of Schneuert and described generators of $A_s$ in \cite{stem}.

An analogous problem for supergroups in the case of positive characteristic has been investigated first by La Scala and Zubkov in \cite{lascala}.
A complete description of generators of polynomial invariants of the adjoint action of the general linear supergroup $G=GL(m|n)$ and generators of $A_s$ was obtained by 
Grishkov, Marko and Zubkov in \cite{grmazu}. 

Since the universal enveloping algebra of a Lie superalgebra $\mathfrak{gl}(m|n)$ is isomorphic to (noncommuting) polynomials, while the universal enveloping algebra of its Cartan subsuperalgebra  $\mathfrak{h}$ is isomorphic to the ring $K[x_1, \ldots, x_m|y_1, \ldots, y_n]$, the invariants of this ring are the correct objects to consider 
when the characteristic of $K$ is zero.

If the characteristic $p$ of the field $K$ is positive, then instead of the universal enveloping algebra of a Lie superalgebra $\mathfrak{gl}(m|n)$ one considers the distribution algebra $Dist(G)$ of $G=GL(m|n)$ and its Frobenius kernels. The basis of the distribution algebra of a maximal torus $T$ of $G$ 
consists of products of binomial elements of the type $\binom{t}{a}= \prod_{i=1}^m \binom{t_i}{a_i}$, where $\binom{x}{a}=\frac{x(x-1)\ldots (x-a+1)}{a!}$, while 
the basis of the distribution algebra of the unipotent subsupergroups $U^+$ and $U^-$, corresponding to unipotent upper and lower triangular matrices of $G$, 
respectively, consists of products of 
divided powers $e^{(b)}=\prod_{i=1}^s e_i^{(b_i)}$, where $x^{(a)}= \frac{x^a}{a!}$.

From now on, the characteristic $p$ of the ground field $K$ is positive. 

\begin{df}
Denote by $Div[x,y]=Div[x_1, \ldots, x_m|y_1, \ldots, y_n]$ the algebra of divided power elements generated by two groups of commuting variables 
$x_1, \ldots, x_m$ and $y_1, \ldots, y_n$. The algebra structure is given via $z^{(a)}z^{(b)}=\binom{a+b}{a}z^{(a+b)}$ for $z=x_1, \ldots, x_m;y_1, \ldots, y_n$.
Thus $Div[x,y]$ is spanned by the products 
$x_1^{(a_1)}\ldots x_m^{(a_m)}$ $y_1^{(b_1)}\ldots y_n^{(b_n)}$, where the exponents $a_i$ and $b_j$ are non-negative integers.  
Define the degree $d$ of the element $x_1^{(a_1)}\ldots x_m^{(a_m)}y_1^{(b_1)}\ldots y_n^{(b_n)}$ to be the sum $\sum_{i=1}^m a_i +\sum_{j=1}^n b_j$. 
Denote by $Div_k[x,y]$ the space of all elements of $Div[x,y]$ of degree $k$. 
\end{df}
It is clear that $Div[x,y]$ is a graded algebra, where the grading is given by the degrees $d$.

We repeatedly use the property of divided powers that $z^{(a)}z^{(b)}=0$ if there is a $p$-adic carry when $a$ is added to $b$; $z^{(a)}z^{(b)}$ is a nonzero scalar multiple of $z^{(a+b)}$ otherwise.

Using linearity, it is easy to verify that 
\[[(\frac{d}{d x_1}+\frac{d}{d y_1})f(x|y)]|_{x_1=y_1=T}=\frac{d}{dT} f (x|y)|_{x_1=y_1=T}.\]
Therefore we can replace the condition $\frac{d}{dT} f (x|y)|_{x_1=y_1=T}=0$ from the definition of the supersymmetric polynomial by the equivalent condition 
\[(\frac{d}{d x_1}+\frac{d}{d y_1})f\equiv 0 \pmod{(x_1-y_1)}.\]
Now we are ready to define supersymmetric elements in $Div[x,y]$.

\begin{df}
Define the derivations $\frac{d}{d x_i}$ and $\frac{d}{d y_j}$ of $Div[x,y]$ via 
\[\frac{d}{d x_i}(g^{(l)})=g^{(l-1)}\frac{d}{d x_i}(g), \frac{d}{d x_i}(x_s)=\delta_{is}, \frac{d}{d x_i}(y_j)=0,\]
and
\[\frac{d}{d y_j}(g^{(l)})=g^{(l-1)}\frac{d}{d y_j}(g), \frac{d}{d y_j}(y_s)=\delta_{js}, \frac{d}{d y_j}(x_i)=0\]
for every $g\in Div[x,y]$, $1\leq i \leq m$ and $1\leq j\leq n$.
If $f\in Div_k[x,y]$ is symmetric in variables $x_1, \ldots, x_m$ 
and $y_1, \ldots, y_n$ separately, and 
there is $f'\in Div_{k-2}[x,y]$ such that
\begin{equation}\label{star}
(\frac{d}{d x_1}+\frac{d}{d y_1})f=(x_1-y_1) f',  \tag{$\star$} 
\end{equation}
then $f$ is called {\it supersymmetric}.
\end{df}

The supersymmetric elements form a graded subalgebra of $Div[x, y]$, which is denoted by $\mathcal{S}$. 
Its homogeneous component of degree $k$ is denoted by $\mathcal {S}_k$.

The structure of the paper is as follows. 
In Section 1, we determine the symmetric elements in the even divided power algebra $Div[x_1, \ldots, x_m]$.
In Section 2, we determine the supersymmetric elements in the divided power algebra $Div[x_1,y_1]$ which serve as a bridge to the general supersymmetric case.
In Section 3, we characterize supersymmetric elements in $Div[x,y]$ using the concepts of marked and unmarked monomials. 
In Section 4, we describe generators of supersymmetric elements in $Div[x_1, \ldots, x_m,y_1]$.
In Sections 5 and 6 we describe generators of supersymmetric elements in $Div[x_1,y_1,y_2]$ and $Div[x_1,x_2, y_1, y_2]$, respectively.

\section{Symmetric elements in the divided power algebra $Div[x_1, \ldots, x_m]$}

Let $\Sigma_m$ be the symmetric group on $m$ elements. It acts on the purely even algebra $Div[x]=Div[x_1, \ldots, x_m]$ of divided powers by permuting variables $x_1, \ldots, x_m$.

In this section, we describe invariants $Div[x]^{\Sigma_m}$, which we call symmetric elements of $Div[x]$.

For an $m$-tuple $\mu=(\mu_1, \ldots, \mu_m)$ of non-negative integers denote its degree $|\mu|=\sum_{i=1}^m \mu_i$. Then every homogeneous element $f\in Div_k[x]$
can be expressed as $f=\sum_{\mu, |\mu|=k}f_{\mu} x^{(\mu)}$.
Order monomials $x^{(\mu)}$ of degree $k$ legicographically.
If $f\in Div_k[x]^{\Sigma_m}$, then the largest  monomial $x^{(\lambda)}$ for which $f_{\lambda}\neq 0$ satisfies 
\[\lambda_1\geq\lambda_2\geq\ldots\geq\lambda_m\geq 0.\] The corresponding monomial $x^{(\lambda)}$ is called the {\it leading term} of $f$.

For $a_1\geq \ldots \geq a_m\geq 0$ denote by $f_{a_1, \ldots, a_m}$ the unique symmetric function 
that has the leading term $x_1^{(a_1)}\ldots x_m^{(a_m)}$ and all other terms of the form 
$x_{\sigma(1)}^{(a_1)} \ldots x_{\sigma(m)}^{(a_m)}$ for some $\sigma\in \Sigma_m$.
For example, 
if $a_1>\ldots >a_m\geq 0$, then 
\[f_{a_1, \ldots, a_m}= \sum_{\sigma\in \Sigma_m} x_{\sigma(1)}^{(a_1)} \ldots x_{\sigma(m)}^{(a_m)}\]

\begin{pr}\label{pr1}
The algebra $Div[x]^{\Sigma_m}$ is generated by the set $D$ of all elements $f_{q,a_2, \ldots, a_m}$, where $q=p^s$ for an arbitrary non-negative integer $s$ and arbitrary integers 
$q\geq a_2\geq \ldots \geq a_m\geq 0$.
The set $D$ is minimal in the sense that none of the generators can be left out.
\end{pr}
\begin{proof}
Denote by $B$ the algebra generated by all elements of $D$. Clearly, $B\subseteq Div[x]^{\Sigma_m}$. The equality $B=Div[x]^{\Sigma_m}$ follows once we show that for every monomial 
$x_1^{(a_1)}\ldots x_m^{(a_m)}$, where $a_1\geq \ldots \geq a_m\geq 0$, 
there is an element of $B$ that has the leading term $x_1^{(a_1)}\ldots x_m^{(a_m)}$. Indeed, take an arbitrary element $g\in Div[x]^{\Sigma_m}$ and assume that its leading term is 
$x_1^{(a_1)}\ldots x_m^{(a_m)}$. There is an element $f$ of $B$ that has the same leading term as $g$.
Then $g-f$ has a smaller leading term, and by induction on the lexicographic order, we can assume that $g-f$ belongs to $B$. Therefore $g\in B$.

If $m=1$, then the set $D$ consists of elements $x_1^{(q)}$, where $q=p^s$ and $s\geq 0$. It is easy to see that $D$ is a minimal generating set of $Div[x_1]$.
Therefore, assume $m>1$.

If $a_1<p$, then, up to a nonzero constant, $x_1^{(a_1)}\ldots x_m^{(a_m)}$ equals 
\[(x_1\ldots x_m)^{a_m}(x_1\ldots x_{m-1})^{a_{m-1}-a_m} \ldots x_1^{a_1-a_2}\] 
which is the leading term of $f_{1, \ldots, 1}^{a_m} f_{1, \ldots, 1, 0}^{a_{m-1}-a_m} \ldots f_{1, 0, \ldots, 0}^{a_1-a_2}\in B$.
Here $f_{\tiny{\underbrace{1, \ldots, 1}_j, 0, \ldots 0}}$ is the $j$-th elementary symmetric function 
$\sigma_{j}=\sum_{1\leq i_1<\ldots < i_j\leq m}x_{i_1}\ldots x_{i_j}$.

Assume now $a_1\geq p$ and write the $p$-adic expansions of $a_i$ as 
\[a_i=\sum_{j=0}^s b_{ij}p^j\]
for each $i=1, \ldots, m$ and sufficiently large $s$. 

If $a_1=\ldots =a_m$, then, up to a nonzero constant, 
$x_1^{(a_1)}\ldots x_m^{(a_m)}$ is the leading term of 
\[f_{p^s, \ldots, p^s}^{b_{1s}} f_{p^{s-1}, \ldots, p^{s-1}}^{b_{1,s-1}}\ldots f_{1, 1, \ldots, 1}^{b_{10}}\in B.\]
Otherwise, there is an index $0\leq t \leq s$ such that
$b_{1s}=\ldots =b_{ms}$, $b_{1,s-1}=\ldots =b_{m,s-1}$,  \ldots  and $b_{r,t} > b_{r+1, t}=b_{r+2,t}=\ldots =b_{m,t}$ for some $r$. In this case,
the product
\[(x_1^{(p^s)}\ldots x_m^{(p^s)})^{b_{1s}}\ldots (x_1^{(p^{t+1})}\ldots x_m^{(p^{t+1})})^{b_{1,t+1}}
(x_1^{(a'_1)}\ldots x_{m}^{(a'_m)}),\]
where $a'_i=\sum_{j=0}^t b_{ij}p^j$ for $i=1, \ldots, m$ are such that $a'_1\geq \ldots \geq a'_m\geq 0$,
is a multiple of $x_1^{(a_1)}\ldots x_m^{(a_m)}$ by a constant $c_1$.
Since there are no $p$-adic carries when adding the exponents in the above product, the constant $c_1$ is not zero.

We can write 
\[\begin{aligned}c_2 x_1^{(a'_1)}\ldots x_{m}^{(a'_m)}=&(x_1^{(p^t)}\ldots x_m^{(p^t)})^{b_{m,t}}(x_1^{(p^t)}\ldots x_{r}^{(p^t)})^{b_{r,t}-b_{r+1,t}-1}\\
&(x_1^{(p^t)}\ldots x_{r}^{(p^t)}x_{r+1}^{(a'_{r+1}-b_{r+1,t}p^t)}\ldots x_{m}^{(a'_{m}-b_{m,t}p^t)})\\
&(x_1^{(p^t)}\ldots x_{r-1}^{(p^t)})^{b_{r-1,t}-b_{r,t}}\ldots (x_1^{(p^t)})^{b_{1t}-b_{2t}}
\end{aligned}\]
for some nonzero constant $c_2$ because there are no $p$-adic carries when adding the exponents on the right-hand side.

Since the exponents of the middle term satisfy 
\[p^t \geq \ldots \geq p^t\geq a'_{r+1}-b_{r+1,t}p^t \geq \ldots \geq a'_{m}-b_{m,t}p^t,\] 
there is an element $f_{p^t, \ldots, p^t, a'_{r+1}-b_{r+1,t}p^t, \ldots, a'_{m}-b_{m,t}p^t}$ of $D$ 
with the leading term $x_1^{(p^t)}\ldots x_{r}^{(p^t)}x_{r+1}^{(a'_{r+1}-b_{r+1,t}p^t)}\ldots x_{m}^{(a'_{m}-b_{m,t}p^t)}$.

The element 
\[\begin{aligned}&\frac{1}{c_1c_2} 
f_{p^s, \ldots, p^s}^{b_{1s}}\ldots f_{p^{t+1}, \ldots, p^{t+1}}^{b_{1,t+1}}
f_{p^t, \ldots, p^t}^{b_{m,t}}
f_{\tiny{\underbrace{p^t, \ldots, p^t}_r, 0, \ldots, 0}}^{b_{r,t}-b_{r+1,t}-1} \\
&f_{\tiny{\underbrace{p^t, \ldots, p^t}_r, a'_{r+1}-b_{r+1,t}p^t, \ldots, a'_{m}-b_{m,t}p^t}}
f_{\tiny{\underbrace{p^t, \ldots, p^t}_{r-1}, 0, \ldots, 0}}^{b_{r-1,t}-b_{r,t}}\ldots f_{p^t,0, \ldots, 0}^{b_{1t}-b_{2t}}
\end{aligned}\]
belongs to $B$ and has the required leading term $x_1^{(a_1)}\ldots x_m^{(a_m)}$.

It is clear that none of the generators corresponding to $q=p^0=1$ can be omitted from the generating set $D$.
The reason why no element corresponding to $q=p^s>1$ can be omitted from the generating set $D$ is because no product of generating elements from $D$ corresponding to powers $p^t$, where $t<s$ can have the leading element starting with $x_1^{(q)}$ because we always encounter a $p$-adic carry when adding exponents at $x_1$. 
\end{proof}

\section{Supersymmetric elements of $Div[x_1,y_1]$}

From now on, as usual within the content of supergroups, we assume that the characteristic $p$ of the ground field $K$ satisfies $p>2$.
In this section, we write $x=x_1$ and $y=y_1$, for simplicity.

First, we determine a $K$-basis of supersymmetric elements of $Div[x,y]$. 
\begin{pr}\label{Kbasis}
The following elements form a $K$-basis of supersymmetric elements in $Div_k[x,y]$ of degree $k=pl+s$, where $0\leq s<p$:  

\begin{itemize}
\item{$x^{(t)}y^{(k-t)}$, where $t\equiv s+1, \ldots, p-1 \pmod p$};
\item{$(s-j)!x^{(k-pr-j)}y^{(pr+j)} -s(s-1)\ldots (j+1) x^{(k-pr-s)}y^{(pr+s)}$, where $j=0, \ldots, s-1$
and $0\leq r \leq l$};
\item{$\sum_{r=0}^l (-1)^r x^{(pr)}y^{(k-pr)}$ if $s=0$}.
\end{itemize}
\end{pr}
\begin{proof}
Let $f=\sum_{i+j=k} a_{ij} x^{(i)}y^{(j)}$ be a homogeneous element of $Div_k[x,y]$. 
Then 
\[(\frac{d}{dx}+\frac{d}{dy})(f)=\sum_{i+j=k} a_{ij} x^{(i-1)}y^{(j)}+\sum_{i+j=k} a_{ij} x^{(i)}y^{(j-1)}=
\sum_{i+j=k-1} (a_{i+1,j}+a_{i,j+1})x^{(i)}y^{(j)}.\]

On the other hand, if
$f'=\sum_{u+v=k-2} b_{uv} x^{(u)}y^{(v)}$, then 
\[\begin{aligned}(x-y)f'=&\sum_{u+v=k-2}(u+1)b_{uv}x^{(u+1)}y^{(v)}-\sum_{u+v=k-2} (v+1)b_{uv} x^{(u)}y^{(v+1)}\\
=&\sum_{i+j=k-1} [ib_{i-1,j}-jb_{i,j-1}]x^{(i)}y^{(j)}.
\end{aligned}\]

Comparing both expressions, we infer that $f$ is supersymmetric if and only if the linear system, consisting of equations
\begin{equation}\tag{i}\label{eq1}
a_{i+1,j}+a_{i,j+1}=ib_{i-1,j}-jb_{i,j-1}\mbox{ where } i+j=k-1
\end{equation}
in variables $a_{0,k}, \ldots, a_{k,0}$, is consistent.

Assume that $s\geq 2$. Then the above system splits into blocks of two different types. 
Blocks of the first type correspond to equations labeled by (i) where $i=k-pr-1, \ldots,k-pr-s$ for fixed $0\leq r\leq l$. The first equation of this block is 
\[a_{k-pr,pr}+a_{k-pr-1,pr+1}=(s-1)b_{k-pr-1,pr}-0b_{k-pr,pr-1}\]
and the last equation is 
\[a_{k-pr-s+1,pr+s-1}+a_{k-pr-s,pr+s}=0b_{k-pr-s-1,pr+s-1}-(s-1)b_{k-pr-s,pr+s}.\]

Row-reducing the augmented matrix of this linear system, we obtain that it is consistent if and only if 
\begin{equation}\label{*}
\sum_{j=0}^{s} \binom{s}{j} a_{k-pr-j,pr+j}=0.
\end{equation}

Blocks of the second type correspond to equations (i) where $i=k-pr-s-1, \ldots, k-(r+1)p$ for $0\leq r<l$. Since the coefficient matrix of this system
is triangular and all coefficients on the main diagonal are nonzero, a system of the second type is always consistent. 

An important observation is that the variable $b_{k-pr-s-1,pr}$ appears only in a block of the second type and 
variable $b_{k-pr-1,pr}$ appears just in a block of the second type. This implies that the original system is consistent if and only if 
all systems corresponding to blocks of the first type are consistent.

Since $a_{t,k-t}$, where $t\equiv s+1, \ldots, p-1 \pmod p$, only appear in a block of the second type, we derive that corresponding element 
$x^{(t)}y^{(k-t)}$ is supersymmetric. We obtain additional generating supersymmetric elements by setting 
\[a_{k-pr-j,pr+j}=(s-j)!, a_{k-pr-s,pr-s}=-s(s-1)\ldots (j+1)\] for $j=0, \ldots s-1$
and each $0\leq r\leq l$ in the equation $(\ref{*})$.

If $s=1$, then each block of the first type consists of a single equation $a_{k-pr,pr}+a_{k-pr-1,pr+1}=0$,
showing that $-x^{(k-pr)}y^{(pr)}+x^{(k-pr-1)}y^{(pr+1)}$, where $0\leq r \leq l$, are additional generating supersymmetric elements.

If $s=0$, then there are only blocks of the first type, but there is an overlap between conditions from different blocks.
Namely, we get equations 
\begin{equation}\label{sup}
a_{k-pr,pr}+a_{k-p(r+1),p(r+1)}=0
\end{equation}
for $0\leq r<l$, which imply that 
$\sum_{r=0}^l (-1)^r x^{(pr)}y^{(k-pr)}$
is supersymmetric.
\end{proof}

The lexicographic order on monomials $x^{(i)}y^{(j)}$ is given as $x^{(i_1)}y^{(j_1)} > x^{(i_2)}y^{(j_2)}$ if $i_1>i_2$,
or $i_1=i_2$ and $j_1>j_2$.
Let $f\in Div[x,y]$ be such that $f=\sum_{i,j} f_{i,j}x^{(i)}y^{(j)}$. We call $x^{(i)}y^{(j)}$ the leading term of $f$ if $f_{i,j}\neq 0$, and 
$f_{k,l}\neq 0$  implies $x^{(i)}y^{(j)}\geq x^{(k)}y^{(l)}$.


Now we describe algebra generators of the algebra $\mathcal{S}$ of supersymmetric elements in $Div[x,y]$.

\begin{pr}\label{gen11}
The algebra $\mathcal{S}$ of supersymmetric elements of $Div[x_1,y_1]$ is generated by elements $\sum_{r=0}^{p^{s-1}} (-1)^r x^{(pr)}y^{(p^s-pr)}$ for $s>0$,
and $xy^{(k-1)}-ky^{(k)}$ for $k>0$. 
\end{pr}
\begin{proof}
Denote by $A$ the algebra generated by elements from the text of the lemma, and by $A_k$ its homogeneous component of degree $k$.
It is clear that $A$ is a subalgebra of $\mathcal{S}$.

We consider $f\in \mathcal{S}_k$ and proceed by the induction on the degree $k$.
It is clear that $\mathcal{S}_1$ is generated by the element $x-y\in A$. Assume now that $\mathcal{S}_u=A_u$ for $u\leq k$. 

By Proposition \ref{Kbasis}, all leading terms of $f\in \mathcal{S}_k$ are of the form $x^{(t)}y^{(k-t)}$, where $0\leq t\leq k$ is not divisible by $p$, together with 
one more term $x^{(k)}$ if $k$ is divisible by $p$. 
Denote by $f_{(t,k-t)}$ an element of $\mathcal{S}_k$ that has $x^{(t)}y^{(k-t)}$ as its leading term. 

If $t\not \equiv 0,p-1 \pmod p$, 
then $(x-y)f_{(t,k-t)}\in \mathcal{S}_1\mathcal{S}_k$ has a leading term $x^{(t+1)}y^{(k-t)}$. If $t=0$, then $A$ contains a generator with the leading term $xy^{(k)}$. 

If $t\geq p$ and $t$ is divisible by $p$,
then there are elements $f_{(1,k-t)}\in \mathcal{S}_{k-t+1}$ and $f_{(t,0)}\in \mathcal{S}_t$ with leading terms $xy^{(k-t)}$ and $x^{(t)}$, respectively. 
Their product $f_{(1,k-t)}f_{(t,0)}\in \mathcal{S}_{k-t+1}\mathcal{S}_{t}$ has the leading term $x^{(t+1)}y^{(k-t)}$. 

Finally, $A$ contains a generator $f_{(p^s,0)}$ with the leading term $x^{(p^s)}$. It $t+1\equiv 0\pmod p$ and $t+1=a_v p^v+\ldots a_1 p$ is the $p$-adic expansion of $t+1$, then 
$f_{(p^v,0)}^{a_v}\ldots f_{(p,0)}^{a_1}\in A$ has the leading term $x^{(t+1)}$. 

It follows from the inductive assumption and from Lemma \ref{Kbasis} that for every $f\in \mathcal{S}_{k+1}$ 
there is an element $a\in A_{k+1}$ that has exactly the same leading term as $f$. Subtracting a suitable scalar multiple of $a$ from $f$ we obtain another element 
$g\in \mathcal{S}_{k+1}$ with the smaller leading term. Induction on the lexicographical order of leading terms of elements in $\mathcal{S}_{k+1}$ concludes the proof.
\end{proof}

\begin{rem}\label{rem0}
It would be interesting to know if the set of generators in Proposition \ref{Kbasis} is minimal; meaning that if any of its element is omitted, 
then the resulting set does not generate $\mathcal{S}$.
Analogously to the arguments in the proof of Proposition \ref{pr1}, we cannot eliminate any generator with the leading term $x^{(p^s)}$ and still generate $\mathcal{S}$.
Also, if $k\leq p$, then no element of the form $xy^{(k-1)}-ky^{(k)}$ can be generated from elements of the form $xy^{(l-1)}-ly^{(l)}$, where $l<k$. Therefore no element 
$xy^{(k-1)}-ky^{(k)}$ for $k\leq p$ can be omitted either. This suggests that the set of generators in Proposition \ref{Kbasis} might be minimal, but further investigation is required.
\end{rem}

\begin{rem}\label{rem1}
Note that in the case of classical supersymmetric polynomials over the field of characteristic zero, the algebra generators are given by $-C_k=xy^{k-1}-y^k$ for $k>0$. 
Since $(k-1)![xy^{(k-1)}-ky^{(k)}]=xy^{k-1}-y^k$, we have an exact correspondence of algebra generators for degrees $k\leq p$. 
If $k>p$, then we replace $-C_k$ with $xy^{(k-1)}-ky^{(k)}$.

To facilitate the transition through the degrees $k$ that are multiples of $p$, we need generators with leading terms $x^{(p^s)}$ for $s>1$. It turns out that is all that is required for  $Div[x,y]$.
\end{rem}

\section{Characterization of supersymmetric elements in $Div[x_1,\ldots, x_m,y_1, \ldots, y_n]$}

For fields $K$ of characteristic zero, generators of supersymmetric polynomials (as presented in \cite{kanttrish}) are given as 
\[C_t=\sum_{i=0}^{min\{t,m\}} (-1)^{t-i} \sigma_i(x)p_{t-i}(y),\]  
where $\sigma_i(x)$ is the $i$-th elementary symmetric polynomial in variables $x_1, \ldots, x_m$ and $p_j(y)$ is the complete $j$-th symmetric polynomial in variables 
$y_1, \ldots, y_n$. Following Remark \ref{rem1}, we first show how to transform the generators $C_t$ to supersymmetric elements of $Div[x, y]$ over a field of characteristic $p$.

We have  
\[p_k(y)=\sum_{j_1+j_2+\ldots+j_n=k} y_1^{j_1}\ldots y_n^{j_n}\] 
and \[y^{j_1}\ldots y_n^{j_n}=(j_1!\ldots j_n!)y^{(j_1)}\ldots y_n^{(j_n)}.\]

Denote by $d(j_1, \ldots, j_n)$ the highest power of $p$ dividing the product $(j_1!\ldots j_n!)$ and by $d_k=min\{d(j_1, \ldots, j_n)|j_1+\ldots+ j_n=k\}$.
It is obvious that the sequence $\{d_k\}_{k=0}^{\infty}$ is nondecreasing. Denote by $\ell_{k}$ the highest term (in the lexicographic order) 
$y_1^{(j_1)}\ldots y_n^{(j_n)}\in Div[y_1, \ldots, y_n]$
such that $\binom{k}{j_1, \ldots, j_n}$ is not divisible by ${p^{d_k+1}}$.

If $t\leq m$, then $C_t$ represents a nonzero supersymmetric element $E_t$ of $Div[x,y]$ and its highest term is $x_1\ldots x_t$.
If $t>m$, then the expression $\frac{C_t}{p^{d_{t-m}}}$ represents a nonzero supersymmetric element $E_t$ of 
$Div[x, y]$ and its highest term with respect to the lexicographic order is $x_1\ldots x_m \ell_{t-m}$.

Next, we characterize the supersymmetric elements in $Div[x,y]$.

To simplify the notation, we write $x^{(i_1, \ldots, i_m)}y^{(j_1, \ldots, j_n)}$ in place of the monomial $x_1^{(i_1)}\ldots x_m^{(i_m)}y_1^{(j_1)}\ldots y_n^{(j_n)}$.

\begin{pr}\label{cond}
A homogeneous element 
\[f=\sum_{i_1+\ldots+i_m+j_1+\ldots+j_n=k} a_{i_1\ldots i_mj_1\ldots j_n} x^{(i_1, \ldots, i_m)}y^{(j_1, \ldots, j_n)}\] of $Div_k[x_1,\ldots, x_m,y_1, \ldots, y_n]$
is supersymmetric if and only if for each $t\geq 0$ such that $i_1+j_1=t=pl+s$, where $0\leq s<p$ and arbitrary 
$i_2, \ldots, i_m,j_2, \ldots, j_m$ such that $i_2+\ldots+i_m+j_2+\ldots+j_n=k-t$ the set of equations
\begin{equation}\label{**}
\sum_{j=0}^s \binom{s}{j} a_{t-pr-j,i_2, \ldots, i_m,pr+j, j_2, \ldots, j_n} =0 \text{ for } s>0
\end{equation}
\begin{equation}\label{sup*}
a_{t-pr,i_2,\ldots, i_m,pr, j_2, \ldots, j_n}+a_{t-p(r+1), i_2, \ldots, i_m,p(r+1), j_2, \ldots, j_n}=0
\end{equation}
for $0\leq r<l$, 
and
\begin{equation}\label{sym}
a_{i_1, \ldots, i_nj_1\ldots j_n}=a_{\sigma(i_1), \ldots, \sigma(i_m), \tau(j_1), \ldots, \tau(j_n)}
\end{equation}
for any permutations $\sigma\in \Sigma_m$ and $\tau\in \Sigma_n$, 
is consistent.
\end{pr}
\begin{proof}
We compute
\[\begin{aligned}(\frac{d}{dx}+\frac{d}{dy})(f)=&\sum_{i_1+\ldots+i_m+j_1+\ldots+j_n=k} a_{i_1\ldots i_mj_1\ldots j_n} 
x^{(i_1-1,i_2, \ldots, i_m)}y^{(j_1, \ldots, j_n)}\\
&+\sum_{i_1+\ldots+i_m+j_1+\ldots+j_n=k} a_{i_1\ldots i_mj_1\ldots j_n} 
x^{(i_1, \ldots, i_m)}y^{(j_1-1,j_2, \ldots, j_n)}\end{aligned}\]
\[=\sum_{i_1+\ldots+i_m+j_1+\ldots+j_n=k-1} (a_{i_1+1,i_2\ldots i_mi_1\ldots j_n} +a_{i_1,\ldots i_m,j_1+1,j_2\ldots j_n})
x^{(i_1, \ldots, i_m)}y^{(j_1, \ldots, j_n)}.
\]

On the other hand, if
\[f'=\sum_{u_1+\ldots+u_m+v_1+\ldots+v_n=k-2} a_{u_1\ldots u_mv_1\ldots v_n} 
x^{(u_1, \ldots, u_m)}y^{(v_1, \ldots, v_n)},\]
then 
\[\begin{aligned}
(x_1-y_1)f'=&\sum_{u_1+\ldots+u_m+v_1+\ldots+v_n=k-2} (u_1+1)b_{u_1\ldots u_mv_1\ldots v_n} 
x^{(u_1+1,u_2 \ldots, u_m)}y^{(v_1, \ldots, v_n)}\\
&+\sum_{u_1+\ldots+u_m+v_1+\ldots+v_n=k-2} (v_1+1)b_{u_1\ldots u_mv_1\ldots v_n} 
x^{(u_1, \ldots, u_m)}y^{(v_1+1,v_2 \ldots, v_n)}\end{aligned}\]
\[=\sum_{i_1+\ldots+i_m+j_1+\ldots+j_n=k-1} (i_1b_{i_1-1,i_2\ldots i_mi_1\ldots j_n} -j_1b_{i_1,\ldots i_m,j_1-1,j_2\ldots j_n})
x^{(u_1, \ldots, u_m)}y^{(v_1, \ldots, v_n)}.
\]

Therefore $f$ satisfies $(\frac{d}{dx}+\frac{d}{dy})(f)=(x_1-y_1)f'$ if and only if the linear system, consisting of equations
\begin{equation}\tag{$i_1j_1$}\label{eq2}
a_{i_1+1,i_2\ldots i_mi_1\ldots j_n} +a_{i_1,\ldots i_m,j_1+1,j_2\ldots j_n}=i_1b_{i_1-1,i_2\ldots i_mi_1\ldots j_n} -j_1b_{i_1,\ldots i_m,j_1-1,j_2\ldots j_n}
\end{equation}
where $i_1+\ldots+i_m+j_1+\ldots+j_n=k-1$ is consistent.

The above equation (\ref{eq2}) differs from the equation (\ref{eq1}) only by inserting "frozen" indices $i_2, \ldots, i_m, j_2, \ldots, j_m$.

Repeating the arguments from the proof of Lemma \ref{Kbasis}, we conclude that the condition $(\frac{d}{dx}+\frac{d}{dy})(f)=(x_1-y_1)f'$ is equivalent to 
the consistency of the system given by the set of equations (\ref{**}) and (\ref{sup*}) involving fixed $t\geq 0$ such that $i_1+j_1=t=pl+s$, where $0\leq s<p$, $0\leq r<l$, and arbitrary "frozen" indices 
$i_2, \ldots, i_m,j_2, \ldots, j_m$ satisfying $i_2+\ldots+i_m+j_2+\ldots+j_n=k-t$.

The claim now follows if we recall that a supersymmetric element is symmetric with respect to indices $x_1, \ldots, x_m$ and $y_1, \ldots, y_n$ separately.
\end{proof}

We call any equation of type (\ref{**}), (\ref{sup*}) or (\ref{sym}) a {\it defining equation}, and a linear system consisting of all equations 
(\ref{**}), (\ref{sup*}) and (\ref{sym}) the {\it defining linear system}.

We would like to extend further the idea presented in the proof of Proposition \ref{gen11} and describe supersymmetric elements in $Div[x, y]$
using the leading terms. 

Let us call a monomial $x^{(i_1,\ldots, i_m)}y^{(j_1,\ldots, j_n)}$ and the corresponding variable \linebreak $a_{i_1, \ldots, i_m, y_1, \ldots, y_n}$ {\it symmetrized} if $i_1\geq i_2\geq \ldots \geq i_m$ and $j_1\geq j_1\geq \ldots \geq j_n$. 
Symmetrized variables are ordered with respect to the lexicographic order. 
Using the action of the symmetric group $\Sigma_m$ on the variables $x_1, \ldots, x_m$ and the action of the symmetric group $\Sigma_n$ on the variables $y_1, \ldots, y_n$, every monomial $x^{(i_1,\ldots i_m)}y^{(j_1,\ldots, j_n)}$ is conjugated to a unique symmetrized monomial  $x^{(u_1,\ldots, u_m)}y^{(v_1, \ldots, v_n)}$, which we call its {\it symmetrization}. In this case we also call $a_{u_1, \ldots, u_m, v_1, \ldots, v_n}$ the symmetrization of $a_{i_1, \ldots, i_m, y_1, \ldots, y_n}$.

If the symmetrization of $x^{(i_1,\ldots, i_m)}y^{(j_1, \ldots, j_n)}$ is a leading term of a supersymmetric element in $Div[x,y]$, 
then we call $x^{(i_1,\ldots, i_m)}y^{(j_1, \ldots, j_n)}$ and the corresponding variable $a_{i_1, \ldots, i_m,j_1,\ldots, j_n}$ 
{\it marked in $Div[x,y]$}. If this is not the case, then we call them {\it unmarked in $Div[x,y]$}.

The next lemma shows a connection between unmarked monomials from different $Div[x,y]$. 

\begin{lm}\label{connect}
Assume that a monomial $\tilde{x}^{(i'_1, \ldots, i'_{m'})}\tilde{y}^{(j'_1,\ldots, j'_{n'})}$ in $Div[\tilde{x}_1, \ldots, \tilde{x}_{m'},\tilde{y}_1, \ldots, \tilde{y}_{n'}]$ is obtained from a monomial
$x^{(i_1, \ldots, i_{m})}y^{(j_1,\ldots, j_{n})}$ in $Div[x_1, \ldots, x_{m},y_1, \ldots, y_{n}]$ by deleting some of the variables $x_i$ and $y_j$.
If the monomial $\tilde{x}^{(i'_1, \ldots, i'_{m'})}\tilde{y}^{(j'_1,\ldots, j'_{n'})}$ is unmarked in $Div[\tilde{x}_1, \ldots, \tilde{x}_{m'},\tilde{y}_1, \ldots, \tilde{y}_{n'}]$, then 
the monomial $x^{(i_1, \ldots, i_{m})}y^{(j_1,\ldots, j_{n})}$ is unmarked in $Div[x_1, \ldots, x_{m},y_1, \ldots, y_{n}]$. 
\end{lm}
\begin{proof}
Let
\[f=\sum_{(i_1, \ldots, i_{m},j_1, \ldots, j_{n})} a_{i_1, \ldots, i_{m},j_1, \ldots, j_{n}} x^{(i_1, \ldots, i_{m})}y^{(j_1, \ldots, j_{n})}\]
be a supersymmetric element in $Div[x_1, \ldots, x_{m},y_1, \ldots, y_{n}]$.
Fix a set of indices $\{k_1, \ldots, k_s\}$ from $\{1, \ldots, m\}$ and a set of indices $\{l_1, \ldots, l_t\}$ from $\{1, \ldots, n\}$.

Write 
\[f= \sum_{(i_{k_1}, \ldots i_{k_s}, j_{l_1}, \ldots j_{l_t})}  x_{k_1}^{(i_{k_1})}\ldots x_{k_s}^{(i_{k_s})} y_{l_1}^{(j_{l_1})} \ldots y_{l_s}^{(j_{l_s})}
f_{(i_{k_1}, \ldots i_{k_s}, j_{l_1}, \ldots j_{l_t})},\]
where 
\[\begin{aligned}f_{(i_{k_1}, \ldots i_{k_s}, j_{l_1}, \ldots j_{l_t})}=&
\sum_{(i_1, \ldots, \widehat{i_{k_1}}, \ldots, \widehat{i_{k_s}}, \ldots, i_{m},j_1, \ldots, \widehat{j_{l_1}}, \ldots, \widehat{j_{l_t}}, \ldots, j_{n})} 
a_{i_1, \ldots, i_{m},j_1, \ldots, j_{n}} \times \\
&x_1^{(i_1)} \ldots \widehat{x_{k_1}^{(i_{k_l})}} \ldots \widehat{x_{k_s}^{(i_{k_s})}} \ldots x_m^{(i_{m})}
y_1^{(j_1)} \ldots \widehat{y_{l_1}^{(j_{l_1})}} \ldots \widehat{y_{l_t}^{(j_{l_t})}} \ldots y_{n}^{(j_n)}.
\end{aligned}\]

By Proposition \ref{cond}, each $f_{(i_{k_1}, \ldots i_{k_s}, j_{l_1}, \ldots j_{l_t})}$ is a supersymmetric element in 
\[Div[x_1, \ldots, \widehat{x_{k_1}}, \ldots, \widehat{x_{k_s}}, \ldots, x_m,y_1, \ldots, \widehat{y_{l_1}}, \ldots, \widehat{y_{l_t}}, \ldots, y_n]=
Div[\tilde{x},\tilde{y}]\]
obtained by ``freezing'' variables $x_{k_i}$ for $i=1, \ldots, s$, $y_{l_j}$ for $j=1, \ldots, t$ and relabeling the remaining variables as $\tilde{x}_{1},\ldots, \tilde{x}_{m-s}$ 
and $\tilde{y}_{1}$, \ldots, $\tilde{y}_{n-t}$.

If the monomial 
\[x_1^{(i_1)} \ldots \widehat{x_{k_1}^{(i_{k_l})}} \ldots \widehat{x_{k_s}^{(i_{k_s})}} \ldots x_m^{(i_{m})}
y_1^{(j_1)} \ldots \widehat{y_{l_1}^{(j_{l_1})}} \ldots \widehat{y_{l_t}^{(j_{l_t})}} \ldots y_{n}^{j_n}=\tilde{x}^{(i'_1, \ldots, i'_{m-s})}\tilde{y}^{(j'_1, \ldots, j'_{n-t})}\]
is unmarked, then for every supersymmetric polynomial $f'$ in $Div[\tilde{x},\tilde{y}]$ such that 
$\tilde{a}_{i'_1, \ldots, i'_{m-s},j'_1, \ldots, j'_{n-t}}\neq 0$ there is a higher $\tilde{a}_{u'_1, \ldots, u'_{m-s},v'_1, \ldots, v'_{n-t}}\neq 0$.
If we apply this to $f_{(i_{k_1}, \ldots, i_{k_s}, j_{l_1}, \ldots, j_{l_t})}$, we obtain that 
$a_{i_1, \ldots, i_{k_1}, \ldots, i_{k_s}, \ldots, i_m,j_1, \ldots, j_{l_1}, \ldots, j_{l_t}, \ldots, j_n}\neq 0$ implies that for some higher coefficient 
$a_{u'_1, \ldots, i_{k_1}, \ldots, i_{k_s}, \ldots, u'_{m-s},v'_1, \ldots, j_{l_1}, \ldots, j_{l_t}, \ldots, v'_{n-t}}\neq 0$.
Therefore $x^{(i_1, \ldots, i_{m})}y^{(j_1,\ldots, j_{n})}$ is unmarked in $Div[x_1, \ldots, x_{m},y_1, \ldots, y_{n}]$. 
\end{proof}

The next lemma describes generators of the algebra $\mathcal{S}$ of all supersymmetric elements in $Div[x,y]$.

\begin{pr}\label{generators}
For each marked symmetrized monomial $x^{(i_1,\ldots, i_m)}y^{(j_1,\ldots, j_n)}$ of $Div[x, y]$ choose a supersymmetric element 
$S(x^{(i_1,\ldots, i_m)}y^{(j_1,\ldots, j_n)})$ in $Div[x, y]$ such that $x^{(i_1,\ldots, i_m)}y^{(j_1, \ldots, j_n)}$ is its leading term.
Denote by $B$ the subalgebra of $\mathcal{S}$ generated by all such elements $S(x^{(i_1,\ldots, i_m)}y^{(j_1,\ldots, j_n)})$. Then $B=\mathcal{S}$.
\end{pr}
\begin{proof}
For a nonzero element $f\in \mathcal{S}$ denote by $\ell(f)=x^{(i_1,\ldots, i_m)}y^{(j_1,\ldots, j_n)}$ its leading term, and by 
$c_{\ell(f)}$ the coefficient of $f$ at $\ell(f)$.

We proceed by induction on the lexicographic order of the leading term $\ell(f)$. Clearly the minimal $f=1$ belongs to $B$. 
Assume that all elements $g\in \mathcal{S}$ with $\ell(g)<\ell(f)$ belong to $B$. Then $f-c_{\ell(f)}S(\ell(f))$ belongs to $\mathcal{S}$ and its leading term 
is smaller than $\ell(f)$, which implies
$f-c_{\ell(f)}S(\ell(f))\in B$. This shows that $f=S(\ell(f))+(f-c_{\ell(f)}S(\ell(f)) \in B$.
\end{proof}

We use the above proposition to describe the algebra $\mathcal{S}$ once all marked and unmarked monomials of $Div[x, y]$ and elements 
$S(x^{(i_1,\ldots, i_m)}y^{(j_1,\ldots, j_n)})$ are determined.
We need to accomplish two different tasks.

The first task is to determine all marked monomials by constructing supersymmetric elements that have these monomials as their leading term.

The second task is to show that every remaining monomial is unmarked. We assume that the coefficient $a_{i_1, \ldots, i_m,j_1, \ldots, j_n}$
appearing in the presentation of a supersymmetric element 
\[f=\sum_{(i'_1, \ldots, i'_m;j'_1,\ldots, j'_n)} a_{i'_1, \ldots, i'_m,j'_1, \ldots, j'_n} x^{(i'_1, \ldots, i'_m)}y^{(j'_1, \ldots, j'_n)}\]
is not zero. Using equations characterizing supersymmetric elements, we need to derive that 
$a_{i'_1, \ldots, i'_m,j'_1, \ldots, j'_n}\neq 0$ for some coefficient higher than $a_{i_1, \ldots, i_m,j_1, \ldots, j_n}$.

In what follows we write \[i_u=pk_u+r_u \text{ and }  j_v=pl_v+s_v,\] where $0\leq r_u,s_v<p$ for $u=1,\ldots, m$ and $v=1, \ldots, n$.

The next lemma plays a vital role in what follows.

\begin{lm}\label{11}
Assume that $p$ divides some of $i_1, \ldots, i_m$ and $j_1>0$. Then the
monomial $x_1^{(i_1,\ldots, i_m)}y_1^{(j_1,\ldots, j_n)}$ is unmarked in 
$Div[x,y]$. 
\end{lm}
\begin{proof}

Assume that $p$ divides $i_t$. It follows from Proposition \ref{Kbasis} that  $x_t^{(i_t)}y_1^{(j_1)}$ in unmarked in $Div[x_t,y_1]$. 
We can also prove it directly by pointing out the equation that involves the variable $a_{i_t, j_1}$ such that all other appearing symmetrized 
variables are higher than $a_{i_t,j_1}$.  

If $0<s_1<p$, then this equation is
\[\sum_{j=0}^{s_1} \binom{s_1}{j} a_{i_t+s_1-j,pl_1+j} =0.\]
If $s_1=0$, then $l_1>0$ and this equation is
\[a_{i_t+p,j_1-p}+a_{i_t,j_1}=0.\]

Adding variables $x_1, \ldots, \widehat{x_r}, \ldots, x_m, y_2, \ldots, y_n$ concludes the claim.
\end{proof}

Let us note that as a special case, if $j_1>0$ and $i_m=0$, then the monomial $x_1^{(i_1,\ldots,i_m)}y_1^{(j_1,\ldots, j_n)}$ is unmarked in $Div[x_1,\ldots, x_m,y_1, \ldots, y_n]$.

To get an intuition about the structure of supersymmetric elements, apply Lemma \ref{cond}, work with symmetrized variables corresponding to $a_{i_1\ldots i_my_1\ldots y_n}$ and look for the free variables given by an echelon form of the matrix of the total linear system given by equations of type (\ref{**}) and (\ref{sup*}).

\begin{lm}\label{p}
For every $k_1>1$, the monomial $x_1^{(pk_1)}$ is marked in $Div[x,y]$.
\end{lm}
\begin{proof}
It is straightforward to verify that the monomial $x_1^{(pk_1)}$ is the leading term of the supersymmetric element 
\[\sum_{\substack{0\leq u_1,\ldots, u_m,v_1, \ldots, v_n\\u_1+\ldots+u_m+v_1+\ldots +v_n=k_1}}  (-1)^{v_1+\ldots +v_n} x^{(u_1p,\ldots,u_mp)}y^{(v_1p,\ldots,v_np)}\]
in $Div[x,y]$.
\end{proof}

\section{Supersymmetric elements in $Div[x_1,\ldots, x_m,y_1]$}

Throughout this section, we assume $m>1$. Let us start with the following lemma.

\begin{lm}\label{span}
Let $m>1$.
Denote by $C$ the subalgebra of $Div[x_1,\ldots, x_m,y_1]$ generated by symmetrized monomials

\begin{itemize}
\item $\rho_i=x_1\ldots x_i$ for $i=1, \ldots m$;  

\item$\rho_{m+j_1}=\rho_m y_1^{(j_1)}=E_{m+j_1}$ for $j_1>0$ (corresponding to generators from characteristic zero case); 

\item $x_1^{(p^s)}$ for $s>1$;

\item $\tau_{i_1,\ldots, i_m,j_1}=x^{(i_1,\ldots,i_m)}y_1^{(j_1)}$, where $i_1+\ldots + i_m$ is divisible by $p$ but none of $i_1, \ldots, i_m$ are divisible by $p$, 
and $j_1\geq 0$. 
\end{itemize}

Then $C$ has a $K$-basis consisting of symmetrized monomials 

\begin{itemize}
\item $x_1^{(k)}$ where $k$ is divisible by $p$;

\item $x^{(i_1,\ldots, i_m)}y_1^{(j_1)}$, where $p$ does not divide $i_1, i_2, \ldots, i_m$ and $j_1\geq 0$.
\end{itemize}
\end{lm}
\begin{proof}
The proof is left to the reader.
\end{proof}

By Lemma \ref{p}, the monomials $x_1^{(p^s)}$ for $s>1$ are marked in $Div[x_1, \ldots, x_m,y_1]$. 
The following lemma shows that the monomials $\rho_i$, $\rho_{m+j_1}$, and $\tau_{i_1,\ldots, i_m,j_1}$ are also marked in $Div[x_1, \ldots, x_m,y_1]$. 
Hence all monomials appearing in Lemma \ref{span} are marked in $Div[x_1, \ldots, x_m,y_1]$.

\begin{pr}\label{jeden}
Every symmetrized monomial $x^{(i_1,\ldots, i_m)}y_1^{(j_1)}$, where $j_1\geq 0$ and none of positive indices $i_1, \ldots, i_m$ is divisible by $p$, 
is marked in $Div[x_1,\ldots, x_m,y_1]$.
\end{pr}
\begin{proof}
Recall the notation $j_1=pl_1+s_1$ and $i_u=pk_u+r_u$ for $u=1,\ldots, m$, where $0\leq r_u,s_1<p$.

Denote by $M_v$ the set of all ordered $v$-tuples of indices $(m_1, \ldots, m_v)$ such that
\[1\leq m_1<\ldots<m_v\leq m, 0<r_{m_1}, \ldots, r_{m_v} \mbox{ and } r_{m_1}+\ldots r_{m_v}+s_1<p \] and by $M$ the union of all such $M_v$.

We show that the element
\[\sum_{v} (-1)^v \sum_{(m_1,\ldots,m_v) \in M_v} \binom{r_{m_{1}}+\ldots +r_{m_v}+s_1}{r_{m_1}, \ldots, r_{m_{v}},s_1}
f_{i_1, \ldots, i_{m_1}-r_{m_1}, \ldots, i_{m_v}-r_{m_v}, \ldots,  i_m}y_1^{(j_1+r_{m_1}+\ldots +r_{m_v})}, \]
that has $$x^{(i_1,\ldots, i_m)}y_1^{(j_1)}$$ as its leading term (corresponding to $v=0$), is supersymmetric in $Div[x_1,\ldots, x_m,y_1]$.

Denote by $N_v$ the set of all variables
\[a_{i_1, \ldots, i_{m_1}-r_{m_1}, \ldots, i_{m_v}-r_{m_v}, \ldots,  i_m, j_1+r_{m_1}+\ldots +r_{m_v}},\]
where $(m_1,\ldots, m_v) \in M_v$, and by $N$ the union of all such $N_v$.

Every defining equation, that involve a variable from $N$, contains 
exactly two variables from $N$.
The variables
\[a_{i_1, \ldots, i_{m_1}-r_{m_1}, \ldots, i_{m_t}-r_{m_t}, \ldots, i_{m_v}-r_{m_v}, \ldots,  i_m, j_1+r_{m_1}+\ldots +r_{m_v}}\] 
and 
\[a_{i_1, \ldots, i_{m_1}-r_{m_1}, \ldots, i_{m_t}, \ldots, i_{m_v}-r_{m_v}, \ldots,  i_m, j_1+r_{m_1}+\ldots+\widehat{r_{m_t}}+\ldots +r_{m_v}},\] 
appear in the equation
\[\sum_{j=0}^{r_{m_1}+\ldots +r_{m_v}+s_1} \binom{r_{m_1}+\ldots +r_{m_v}+s_1}{j} \times\]
\[a_{i_1, \ldots, i_{m_1}-r_{m_1}, \ldots, i_{m_t}+r_{m_1}+\ldots+\widehat{r_{m_t}}+\ldots +r_{m_v}+s_1-j, \ldots, i_{m_v}-r_{m_v},\ldots,  i_m, j_1-s_1+j}=0
\] 
corresponding to the values $j=r_{m_1}+\ldots + r_{m_v}+s_1$ and $j=r_{m_1}+\ldots + \widehat{r_{m_t}}+\ldots+ r_{m_v}+s_1$
where $(m_1, \ldots, \widehat{m_t}, \ldots, m_v)\in M_{v-1}$ is obtained from $(m_1, \ldots, m_v)\in M_v$ by removing $m_t$.

Assume $(m'_1, \ldots, m'_t, \ldots, m'_{v+1})\in M_{v+1}$ is obtained from $(m_1, \ldots, m_v)\in M_v$ by inserting $m'_t$.
That means that $m'_l=m_l$ for $l<t$ and $m'_{l+1}=m_l$ for $l\geq t$. Then the variables
\[a_{i_1, \ldots, i_{m_1}-r_{m_1}, \ldots, i_{m'_t}, \ldots, i_{m_{v}}-r_{m_{v}}, \ldots,  i_m, r_{m_1}+\ldots +r_{m_v}+j_1}\]
and 
\[a_{i_1, \ldots, i_{m'_1}-r_{m'_1}, \ldots, i_{m'_t}-r_{m'_t}, \ldots, i_{m'_{v+1}}-r_{m'_{v+1}}, \ldots,  i_m, r_{m'_1}+\ldots +r_{m'_{v+1}}+j_1}\] 
appear in the equation
\[\sum_{j=0}^{r_{m'_1}+\ldots +r_{m'_{v+1}}+s_1} \binom{r_{m'_1}+\ldots +r_{m'_{v+1}}+s_1}{j} \times\]
\[a_{i_1, \ldots, i_{m'_1}-r_{m'_1}, \ldots, i_{m'_t}+r_{m'_1}+\ldots \widehat{r_{m'_t}} +r_{m'_{v+1}}+s_1-j, \ldots, i_{m'_{v+1}}-r_{m'_{v+1}}, \ldots,  i_m, j_1-s_1+j}=0\] 
corresponding to the values $j=r_{m_1}+\ldots + r_{m_v}+s_1$ and $j=r_{m'_1}+\ldots +\ldots+ r_{m'_{v+1}}+s_1$.

Since there are no other equations involving variables from $N$, we can set the value of every variable not in $N$ equal to zero. Then the linear system turns into
a smaller one, consisting of equations
\[\begin{aligned}
&a_{i_1, \ldots, i_{m_1}-r_{m_1}, \ldots, i_{m_t}-r_{m_t}, \ldots, i_{m_v}-r_{m_v},\ldots,  i_m, r_{m_1}+\ldots +r_{m_v}+j_1}\\
&+\binom{r_{m_1}+\ldots +r_{m_v}+s_1}{r_{m_t}} 
a_{i_1, \ldots, i_{m_1}-r_{m_1}, \ldots, i_{m_t}, \ldots, i_{m_v}-r_{m_v},\ldots,  i_m, r_{m_1}+\ldots+\widehat{r_{m_t}}+\ldots +r_{m_v}+j_1}=0\\
\end{aligned}\] 
\[\begin{aligned}&
\binom{r_{m'_1}+\ldots +r_{m'_{v+1}}+s_1}{r_{m'_t}} a_{i_1, \ldots, i_{m_1}-r_{m_1}, \ldots, i_{m'_t}, \ldots, i_{m_v}-r_{m_v},\ldots,  i_m, r_{m_1}+\ldots+\ldots +r_{m_v}+j_1}\\
&+a_{i_1, \ldots, i_{m'_1}-r_{m'_1}, \ldots, i_{m'_t}-r_{m'_t}, \ldots, i_{m'_{v+1}}-r_{m'_{v+1}},\ldots,  i_m, r_{m'_1}+\ldots +r_{m'_{v+1}}+j_1}=0
\end{aligned}\] 
for all $(m_1, \ldots, m_v)\in M_v$ and $(m'_1,\ldots, m'_{v+1})\in M_{v+1}$ as above.

Since
\[\binom{r_{m_1}+\ldots +r_{m_v}+s_1}{r_{m_t}}\binom{r_{m_1}+\ldots +\widehat{r_{m_t}}+\ldots +r_{m_v}+s_1}{r_{m_1}, \ldots, \widehat{r_{m_t}}, \ldots, r_{m_v},s_1}=
\binom{r_{m_1}+\ldots +r_{m_v}+s_1}{r_{m_1}, \ldots, r_{m_v},s_1}\]
and
\[\binom{r_{m'_1}+\ldots +r_{m'_{v+1}}+s_1}{r_{m'_t}}
\binom{r_{m'_1}+\ldots \widehat{r_{m'_t}}+\ldots + r_{m'_{v+1}}+s_1}{r_{m'_1}, \ldots, \widehat{r_{m'_t}},\ldots, r_{m'_{v+1}},s_1}=
\binom{r_{m'_1}+\ldots +r_{m'_{v+1}}+s_1}{r_{m'_1}, \ldots, r_{m'_{v+1}},s_1},\]
by setting
\[a_{i_1, \ldots, i_{m_1}-r_{m_1}, \ldots, i_{m_v}-r_{m_v},\ldots,  i_m, j_1+r_{m_1}+\ldots +r_{m_v}}=
(-1)^v \binom{r_{m_1}+\ldots +r_{m_v}+s_1}{r_{m_1}, \ldots, r_{m_v},s_1}\]
we obtain a solution of the defining linear system and the corresponding supersymmetric element.
\end{proof}

In the next proposition we consider monomials in $Div[x_1,x_2,y_1]$. 
The result then extends to monomials in $Div[x_1,\ldots, x_m,y_1, \ldots, y_n]$.

\begin{pr}\label{5.5}
Every symmetrized monomial 
$x^{(i_1,i_2)}y_1^{(j_1)}$ such that $i_2>0$ and $p$ divides $i_1$ or $i_2$
is unmarked in $Div[x_1,x_2,y_1]$.
\end{pr}
\begin{proof}
If $j_1>0$, then the statement follows from Lemma \ref{11}. Therefore assume $j_1=0$.

If $r_2=0$, then $k_2>0$ and $a_{i_1,i_2,0}=(-1)^{k_2} a_{i_1,0,i_2}= a_{i_1+i_2,0,0}$ which is higher than $a_{i_1,i_2,0}$, showing that $a_{i_1,i_2,0}$ is unmarked.

If $r_2>0$, then the equation 
\[\sum_{j=0}^{r_2} \binom{r_2}{j} a_{pk_2+r_2-j,i_1,j} =0\]
shows that $a_{i_1, i_2, 0}\neq 0$ implies 
$a_{i_1,pk_2+r_2-j,j}=a_{pk_2+r_2-j,i_1,j}\neq 0$ for some $j>0$. 

The equation 
\[\sum_{i=0}^j \binom{j}{i} a_{pk_1+j-i,pk_2+r_2-j,i} =0\]
shows that in this case $a_{pk_1+j-i,pk_2+r_2-j,i}\neq 0$ for some $i<j$. 
Since $a_{pk_1+j-i,pk_2+r_2-j,i}$ is higher than $a_{i_1, i_2, 0}$, the claim follows.
\end{proof}

\begin{cor}\label{cor2}
Every symmetrized monomial 
$x_1^{(i_1,\ldots,i_m)}y_1^{(j_1,\ldots,j_n)}$
such that $p$ divides $i_r>0$ for some $r>1$, or $p$ divides $i_1$ and $i_2>0$,
is unmarked in $Div[x,y]$.
\end{cor}
\begin{proof}
Use Lemmas \ref{5.5} and \ref{connect}.
\end{proof}

\begin{theorem}
The algebra $A$ of all supersymmetric elements in $Div[x_1, \ldots, x_m,y_1]$ is generated by elements $S(M)$, where $M$ is one of the symmetrized monomials 
\begin{itemize}
\item $\rho_i=x_1\ldots x_i$ for $i=1, \ldots m$;  

\item$\rho_{m+j_1}=\rho_m y_1^{(j_1)}=E_{m+j_1}$ for $j_1>0$ (corresponding to generators from characteristic zero case); 

\item $x_1^{(p^s)}$ for $s>1$;

\item $\tau_{i_1,\ldots, i_m,j_1}=x_1^{(i_1,\ldots,i_m)}y_1^{(j_1)}$, where $i_1+\ldots + i_m$ is divisible by $p$ but none of $i_1, \ldots, i_m$ are divisible by $p$, 
and $j_1\geq 0$. 
\end{itemize}
\end{theorem}
\begin{proof}
Lemma \ref{p} shows that every monomial $x_1^{(k)}$, where $k$ is divisible by $p$ is marked. Proposition \ref{jeden} shows that every symmetrized 
monomial $x^{(i_1,\ldots, i_m)}y_1^{(j_1)}$, where $p$ does not divide any positive $i_1, i_2, \ldots, i_m$ and $j_1\geq 0$, is also marked. 
On the other hand, Corollary \ref{cor2} shows that all remaining monomials are unmarked. 
Since all monomials $M$ from the text of the theorem are marked,
Lemma \ref{span} shows that each marked monomial is the leading term of some element of the algebra $A$. 
Proposition \ref{generators} concludes the proof.
\end{proof}

The following lemma is needed later.

\begin{lm}\label{blow}
Assume that $f=\sum_{(i_1,i_2,j_1)} a_{(i_1,i_2,j_1)}x^{(i_1,i_2)}y^{(j_1)}$ is a supersymmetric element in $Div[x_1,x_2,y_1]$.
If $p$ divides $j_1$ and $i_2$, and $i_1\geq p$, then 
$a_{i_1,i_2,j_1}=a_{i_1-p,i_2+p,j_1}=\ldots=a_{s_1,i_1+i_2-s_1,j_1}$.
\end{lm}
\begin{proof}
It is enough to show $a_{i_1,0,i_2+j_1}=a_{s_1,i_1-s_1,i_2+j_1}$.

We consider the following system of equations:
\begin{equation}\tag{$e_{s_1}$}
\sum_{i=0}^{s_1} \binom{s_1}{i} a_{i_1-i,0,i_2+j_1+i}=0
\end{equation}
\begin{equation}\tag{$e_{s_1-1}$}
\sum_{i=0}^{s_1-1} \binom{s_1-1}{i} a_{i_1-1-i,1,i_2+j_1+i}=0
\end{equation}
\[\ldots\]
\begin{equation}\tag{$e_1$}
\sum_{i=0}^1 a_{i_1-s_1+1-i,s_1-1,i_2+j_1+i}=0
\end{equation}
\begin{equation}\tag{$e'_{s_1}$}
\sum_{i=0}^{s_1} \binom{s_1}{i} a_{s_1-i,i_1-s_1,i_2+j_1+i}=0
\end{equation}
\begin{equation}\tag{$e'_{s_1-1}$}
\sum_{i=0}^{s_1-1} \binom{s_1-1}{i} a_{s_1-1-i,i_1-s_1+1, i_2+j_1+i}=0
\end{equation}
\[\ldots\]
\begin{equation}\tag{$e'_1$}
\sum_{i=0}^1 a_{1-i,i_1-1,i_2+j_1+i}=0
\end{equation}

When we expand the sum $\sum_{i=1}^{s_1-1} \binom{s_1}{i} a_{i_1-i,0,i_2+j_1+i}$ using equations $(e'_{s_1-1})$ through $(e'_1)$
we observe that the term $a_{s_1-l-i,i_1-s_1+l, i_2+j_1+i}$, where $1\leq l\leq s_1-1$ and $0\leq i<s_1-l$ appears with the coefficient $-\binom{s_1-l}{i}\binom{s_1}{s_1-l}$.

When we expand the sum $\sum_{i=1}^{s_1-1} \binom{s_1}{i} a_{s_1-i,i_1-s_1,i_2+j_1+i}$ using equations $(e_{s_1-1})$ through $(e_1)$
we observe that the term $a_{i_1-s_1+l,s_1-l-i,i_2+j_1+i}$, where $1\leq l\leq s_1-1$ and $0\leq i<s_1-l$ appears with the coefficient $-\binom{l+i}{i}\binom{s_1}{l+i}$.

Since $\binom{s_1-l}{i}\binom{s_1}{s_1-l}=\frac{s_1!}{i!(s_1-l-i)!l!}=\binom{l+i}{i}\binom{s_1}{l+i}$, we obtain 
\[\sum_{i=1}^{s_1-1} \binom{s_1}{i} a_{i_1-i,0,i_2+j_1+i}= \sum_{i=1}^{s_1-1} \binom{s_1}{i} a_{s_1-i,i_1-s_1,i_2+j_1+i}.\]
This together with 
\[\sum_{i=0}^{s_1} \binom{s_1}{i} a_{i_1-i,0,i_2+j_1+i}= \sum_{i=0}^{s_1} \binom{s_1}{i} a_{s_1-i,i_1-s_1,i_2+j_1+i}=0,\]
implies $a_{i_1,0,i_2+j_1}=a_{s_1,i_1-s_1,i_2+j_1}$.
\end{proof}

\begin{rem}
It is not true that for each supersymmetric element \linebreak $f=\sum_{(i_1,i_2,j_1)} a_{(i_1,i_2,j_1)}x^{(i_1,i_2)}y^{(j_1)}$ 
the conditions $i_1\equiv u_1 \pmod p$, $i_2\equiv u_2 \pmod p$ imply $a_{i_1,i_2,0}= a_{u_1,u_2,0}$.
As a counterexample, there is a supersymmetric element $f$ for which $a_{7,1,0}\neq a_{4,4,0}$ and $p=3$.
\end{rem}

\section{Supersymmetric elements in $Div[x_1,y_1,y_2]$}

Throughout this section, we assume that $n>1$. Let us introduce the concept of the height of $n$-tuple $(y_1, \ldots, y_n)$ in a form suitable to a more general setting
of $Div[x,y]$. 

\begin{df}
Let $j_1\geq \ldots \geq j_n\geq 0$. If 
\begin{itemize}
\item $j_1\leq p-1$ and for every $1\leq t \leq n$ such that $s_t\neq p-1$ we have $j_{t+1}=0$ and $s_u=p-1$ for $1\leq u<t$, or 
\item $j_1>p-1$ and $s_1=\ldots = s_{n-1}=p-1$,
\end{itemize}
then we say that $(j_1,\ldots, j_n)$ has the height $h=1$; otherwise we say that it has the height $h>1$.
\end{df}

For $Div[x_1,y_1,y_2]$ we have $n=2$ and a simpler description - $(j_1,j_2)$ is of the height one if and only if either $0<j_1<p-1$ and $j_2=0$, or $s_1=p-1$.

\begin{lm}\label{span2}
Denote by $C$ the subalgebra of $Div[x_1,y_1,y_2]$ generated by symmetrized monomials 
\begin{itemize}
\item{$x_1^{(p^s)}$ for $s>0$;}
\item{$x_1y^{(j_1,j_2)}$ for  $(j_1,j_2)$ of height $h=1$;}
\item{$x_1^{(2)}y^{(j_1,j_2)}$ for $(j_1,j_2)$ of height $h>1$.}
\end{itemize}

Then $C$ has a $K$-basis consisting of symmetrized monomials
\begin{itemize}
\item $x_1^{(i_1)}$, where $r_1=0$;
\item $x_1^{(i_1)}y^{(j_1,j_2)}$, where $r_1\geq 1$ and $(j_1,j_2)$ has height $h=1$;
\item $x_1^{(i_1)}y^{(j_1,j_2)}$, where $r_1\geq 2$ and $(j_1,j_2)$ has height $h>1$.
\end{itemize}
\end{lm}
\begin{proof}
The proof is left to the reader.
\end{proof}

The next statement is formulated for a more general case $Div[x_1, y_1, \ldots, y_n]$ instead of $Div[x_1,y_1,y_2]$.

\begin{lm}\label{6.2}
If $(j_1,\ldots, j_n)$ has the height $h=1$, then every symmetrized monomial 
$x_1y^{(j_1,\ldots, j_n)}$ is marked in $Div[x_1,y_1,\ldots, y_n]$.
\end{lm}
\begin{proof}
If $j_1\leq p-1$ and for every $1\leq t \leq n$ such that $s_t\neq p-1$ we have $j_{t+1}=0$ and $s_u=p-1$ for $1\leq u<t$, then 
$x_1y_1^{(j_1,\ldots, j_n)}$ is marked since it is the leading term of $E_{1+j_1+\ldots+j_n}$.

Therefore we can assume $s_1=\ldots=s_{n-1}=p-1$. If $s_n=p-1$, then $a_{1,j_1,\ldots, j_n}$ does not appear in any defining equation and is therefore marked.
If $s_n\neq p-1$, then we have a series of $p-1-s_n$ pairs of equations for each choice of $t=1, \ldots, n-1$:
\[a_{s_2+1,j_n-s_n,j_1, \ldots, j_{n-1}} + \ldots + (s_n+1) a_{1,j_n,j_1,\ldots, j_{n-1}} + a_{0,j_n+1,j_1, \ldots, j_{n-1}}=0\]
\[a_{p-1,j_t-p+1,j_n+1,j_1, \ldots, \widehat{j_t}, \ldots, j_{n-1}} + \ldots + (p-1) a_{1,j_t-1,j_n+1, j_1, \ldots, \widehat{y_t}, \ldots, j_{n-1}} + 
a_{0,j_t,j_n+1,j_1, \ldots, \widehat{j_t}, \ldots, j_{n-1}}=0,\]
\[a_{s_n+2,j_n-s_n,j_t-1,j_1, \ldots, \widehat{j_t}, \ldots, j_{n-1} } + \ldots + (s_n+2) a_{1,j_n+1,j_t-1,j_1, \ldots, \widehat{j_t}, \ldots, j_{n-1}} 
+ a_{0,j_n+2,j_t-1,j_1, \ldots, \widehat{j_t}, \ldots, j_{n-1}}=0\]
\[a_{p-2,j_t-p+1,j_n+2,j_1, \ldots, \widehat{j_t}, \ldots, j_{n-1}} + \ldots + (p-2) a_{1,j_t-2,j_n+2,j_1, \ldots, \widehat{j_t}, \ldots, j_{n-1}} 
+ a_{0,j_t-1,j_n+2,j_1, \ldots, \widehat{j_t}, \ldots, j_{n-1}}=0,\]
\[\ldots\]
\[\ldots\]
\[\begin{aligned}
&a_{p-1,j_n-s_n,j_t-p+s_n+2,j_1, \ldots, \widehat{j_t}, \ldots, j_{n-1}} + \ldots + (p-1) a_{1,j_n+p-s_n-2,j_t-p+s_n+2,j_1, \ldots, \widehat{j_t}, \ldots, j_{n-1}} \\
&+ a_{0,j_n+p-s_n-1,j_t-p+s_n+2,j_1, \ldots, \widehat{j_t}, \ldots, j_{n-1}}=0
\end{aligned}\]
\[\begin{aligned}
&a_{s_n+1,j_t-p+1,j_n+p-s_n-1,j_1, \ldots, \widehat{j_t}, \ldots, j_{n-1}} + \ldots + (s_n+1) a_{1,j_t-p+s_n+1,j_n+p-s_n-1,j_1, \ldots, \widehat{j_t}, \ldots, j_{n-1}} \\
&+ a_{0,j_t-p+s_n+2,j_n+p-s_n-1,j_1, \ldots, \widehat{j_t}, \ldots, j_{n-1}}=0.
\end{aligned}\]
These are the only defining equations involving variables $a_{1,j_n,j_1, \ldots, j_{n-1}}$, $a_{0,j_n+1, j_1, \ldots, j_{n-1}}$,
\[a_{1,j_n+1,j_t-1, j_1, \ldots, \widehat{j_t}, \ldots, j_{n-1}}, \ldots, a_{1,j_n+p-s_n-1,j_t-p+s_n+1,j_1, \ldots, \widehat{j_t}, \ldots, j_{n-1}}\] and 
\[a_{0,j_t-1, j_n+2,j_1, \ldots, \widehat{j_t}, \ldots, j_{n-1}}, \ldots, a_{0,j_t-p+s_n+2,j_n+p-s_n-1,j_1, \ldots, \widehat{j_t}, \ldots, j_{n-1}}\]
for all $t=1, \ldots, n-1$.
When we set all variables that are not listed above to zero and $a_{1,j_1,j_2,\ldots, j_n}=1$, the values of the remaining variables are
\[a_{0,j_n+1, j_1, \ldots, j_{n-1}}=-(s_n+1),\]
\[a_{1,j_n+j,j_t-j,j_1, \ldots, \widehat{j_t}, \ldots, j_{n-1}}=(-1)^j\binom{s_n+j}{j}\]
for $j=1, \ldots, p-1-s_n$ and 
\[a_{0,j_t-j,j_n+j+1,j_1, \ldots, \widehat{j_t}, \ldots, j_{n-1}}=(-1)^{j+1}(s_n+j+1)\binom{s_n+j}{j}\] 
for $j=1, \ldots, p-2-s_n$.

Denote by $g_{j_1,\ldots, j_n}$ the symmetric function analogous to $f_{i_1, \ldots, i_m}$, defined on $y_1, \ldots, y_n$ instead of $x_1,\ldots, x_m$.

Then
\[\begin{aligned}
&x_1 g_{j_1, \ldots, j_n}-(s_n+1)g_{j_1,j_2+1,j_3,\ldots, j_n}+
\sum_{t=1}^{n-1}\sum_{j=1}^{p-1-s_n} (-1)^j\binom{s_n+j}{j} x_1 g_{j_n+j,j_t-j,j_1, \ldots, \widehat{j_t}, \ldots, j_{n-1}} \\
&+ \sum_{t=1}^{n-1}\sum_{j=1}^{p-2-s_n} (-1)^{j+1}(s_n+j+1)\binom{s_n+j}{j} g_{j_t-j,j_n+j+1,j_1, \ldots, \widehat{j_t}, \ldots, j_{n-1}},
\end{aligned}\]
is a supersymmetric element that has the leading term $x_1y^{(j_1,\ldots, j_n)}$, showing that $a_{1,j_1,\ldots, j_n}$ is marked.
\end{proof}

For the next three lemmas, we work inside $Div[x_1,y_1,y_2]$.

\begin{lm}\label{6.3}
Every symmetrized monomial 
$x_1^{(2)}y^{(j_1,j_2)}$ is marked in $Div[x_1,y_1,y_2]$.
\end{lm}
\begin{proof}
The statement is true for $(j_1,j_2)$ of height $h=1$ by Lemma \ref{6.2}.


First, assume that $j_1>j_2$.
Then the element 
\[x_1^{(2)}g_{j_1,j_2} - \frac{s_1+1}{2}x_1g_{j_1+1,j_2} -\frac{s_2+1}{2} g_{j_1,j_2+1} 
+\frac{(s_1+1)(s_2+1)}{2} g_{j_1+1,j_2+1}\]
is a supersymmetric element with the leading term $x_1^{(2)}y_1^{(j_1)}y_2^{(j_2)}$, showing that $a_{2j_1j_2}$ is marked.
This is easy to verify since the variables corresponding to summands in this expression appear just in the following 
defining equations 
\begin{equation}\tag{$A_t$}
a_{s_1+2-t,j_1-s_1,j_2+t} + \ldots +\binom{s_1+2-t}{2}a_{2,j_1-t,j_2+t}+(s_1+2-t)a_{1,j_1+1-t,j_2+t}+a_{0,j_1+2-t,j_2+t}=0
\end{equation}
for $t=0,1$
and 
\begin{equation}\tag{$B_t$}
a_{s_2+2-t,j_1+t,j_2-s_2} + \ldots +\binom{s_2+2-t}{2}a_{2,j_1+t,j_2-t}+(s_1+2-t)a_{1,j_1+t,j_2+1-t}+a_{0,j_1+t,j_2+2-t}=0
\end{equation}
for $t=0,1$.

If $j_1=j_2=j$, then
\[x_1^{(2)}g_{j,j} - \binom{s+2}{2} g_{j+2,j}\] 
is a supersymmetric element with the leading term $x_1^{(2)}y^{(j,j)}$ showing that $a_{2jj}$ is marked. 
This is easy to verify this since the variables corresponding to summands in this expression appear only in the 
defining equations $(A_0)=(B_0)$ and $(A_1)=(B_1)$. 
\end{proof}

\begin{lm}\label{blowout}
Assume that $f=\sum_{(i_1,j_1,j_2)} a_{(i_1,j_1,j_2)}x^{(i_1)}y^{(j_1,j_2)}$ is a supersymmetric element in $Div[x_1,y_1,y_2]$.
If $r_1=1$, $s_1<p-1$, $l_1>0$ and $s_2=0$, then $a_{i_1,j_1,j_2}\neq 0$ implies that some $a_{u_1,v_1,v_2}\neq 0$, where $u_1>i_1$.
\end{lm}
\begin{proof}
First assume $l_2>0$. Then the claim follows by considering the system of equations
\[\begin{aligned}&a_{i_1+s_1,j_1-s_1,j_2} &+ \ldots + &(s_1+1) a_{i_1,j_1,j_2} &+ a_{i_1-1,j_1+1,j_2}&=0\\
& & &a_{i_1-1+p,j_2-p,j_1+1}&+a_{i_1-1,j_2,j_1+1}&=0.
\end{aligned}\] 

If $l_2=0$, then $j_2=0$. Considering the series of $s_1+1$ pairs of equations
\[\begin{aligned}
&&&a_{i_1,0,j_1} &+a_{i_1-1,1,j_1}&=0\\
&a_{i_1+s_1-1,j_1-s_1,1} +&\ldots &+s_1 a_{i_1,j_1-1,1} &+a_{i_1-1,j_1,1}&=0\\
&&a_{i_1+1,0,j_1-1}&+2a_{i_1,1,j_1-1}&+a_{i_1-1,2,j_1-1}&=0\\
&a_{i_1+s_1-2,j_1-s_1,2} +&\ldots &+(s_1-1) a_{i_1,j_1-2,2} &+a_{i_1-1,j_1-1,2}&=0\\
&&\ldots&&&\\
&&\ldots&&&\\
&a_{i_1+s_1-1,0,j_1-s_1+1}+&\ldots &+s_1a_{i_1,s_1-1,j_1-s_1+1}&+a_{i_1-1,s_1,j_1-s_1+1}&=0\\
&&&a_{i_1,j_1-s_1,s_1} &+a_{i_1-1,j_1-s_1+1,s_1}&=0\\
&a_{i_1+s_1,0,j_1-s_1}+&\ldots &+(s_1+1)a_{i_1,s_1,j_1-s_1}&+a_{i_1-1,s_1+1,j_1-s_1}&=0\\
&&&a_{i_1-1+p,j_1-s_1-p,s_1+1} &+a_{i_1-1,j_1-s_1,s_1+1}&=0
\end{aligned}\]
implies the claim.
\end{proof}

\begin{lm}\label{6.5}
If $r_1=1$ and $(j_1,j_2)$ has the height $h>1$, then the symmetrized monomial $x_1^{(i_1)}y^{(j_1,j_2)}$ is unmarked in $Div[x_1,y_1,y_2]$.
\end{lm}
\begin{proof}
Since $(j_1,j_2)$ has height $h>1$, we have $0\leq s_1<p-1$.
If $s_2>0$, then the equations 
\[a_{i_1+s_1,j_1-s_1,j_2} + \ldots + (s_1+1) a_{i_1,j_1,j_2} + a_{i_1-1,j_1+1,j_2}=0\]
\[a_{i_1+s_2-1,j_2-s_2,j_1+1} + \ldots + s_2 a_{i_1,j_2-1,j_1+1} + a_{i_1-1,j_2,j_1+1}=0\]
imply that $a_{i_1,j_1,j_2}$ is unmarked since $a_{i_1-1,j_2,j_1+1}$ is a linear combination of terms higher than $a_{i_1,j_1,j_2}$. 

If $s_2=0$ and $l_2>0$, then $a_{i_1,j_1,j_2}$ is unmarked by Lemma \ref{blowout}.

Finally, if $j_2=0$, then $j_1>p-1$ and $a_{i_1,j_1,0}$ is unmarked by Lemma \ref{blowout}.
\end{proof}

\begin{cor}\label{cor3}
If $r_1=1$ and $(j_1,\ldots, j_n)$ has the height $h>1$, then every symmetrized monomial $x_1^{(i_1)}y^{(j_1,\ldots,j_n)}$ is unmarked in $Div[x_1,y_1,\ldots, y_n]$.
\end{cor}
\begin{proof}
If $(j_1, \ldots, j_n)$ has the height $h>1$, then there is a pair $(j_u,j_v)$, obtained by removing entries from $(j_1, \ldots, j_n)$, that has the height $h>1$.
Use Lemmas \ref{6.5} and \ref{connect} to complete the proof.
\end{proof}

\begin{theorem}
The algebra $A$ of all supersymmetric elements in $Div[x_1,y_1,y_2]$ is generated by elements $S(M)$, where $M$ is one of the symmetrized monomials 
\begin{itemize}
\item{$x_1^{(p^s)}$ for $s>0$;}
\item{$x_1y^{(j_1,j_2)}$ for  $(j_1,j_2)$ of height $h=1$;}
\item{$x_1^{(2)}y^{(j_1,j_2)}$ for $(j_1,j_2)$ of height $h>1$.}
\end{itemize}
\end{theorem}
\begin{proof}
Lemmas \ref{span2}, \ref{p}, \ref{6.2} and \ref{6.3} describe all marked monomials.
Lemma \ref{11} and Corollary \ref{cor3} describe all unmarked monomials. Lemma \ref{span2} and Proposition \ref{generators} conclude the proof.
\end{proof}



\section{Supersymmetric elements in $Div[x_1,x_2,y_1,y_2]$} 

\begin{lm}\label{span3}
Denote by $C$ the subalgebra of $Div[x_1,x_2,y_1,y_2]$ generated by following symmetrized monomials:
\begin{itemize}
\item{$x_1$};
\item{$x_1^{(p^s)}$ for some $s>0$};
\item{$x_1x_2y^{(j_1,j_2)}$, where $(j_1,j_2)$ has the height $h=1$;}
\item{$x^{(i_1,i_2)}y^{(j_1,j_2)}$, where $r_1=r_2=1$, $k_2>0$ and $s_1=p-1$;}
\item{$x^{(i_1,i_2)}y^{(j_1,j_2)}$, where $r_1=r_2=2$.}
\end{itemize}
Then $C$ has a $K$-basis consisting of the following monomials 
\begin{itemize}
\item $x_1^{(i_1)}$;
\item $x_1^{(i_1)}x_2y^{(j_1,j_2)}$, where $r_1>0$ and $(j_1,j_2)$ has the height $h=1$;
\item $x^{(i_1,i_2)}y^{(j_1,j_2)}$, where $r_1>0$, $r_2=1$, $k_2>0$ and $s_1=p-1$;
\item $x^{(i_1,i_2)}y^{(j_1,j_2)}$, where $r_1,r_2\geq 2$.
\end{itemize}
\end{lm}
\begin{proof}
The proof is left for the reader.
\end{proof}

\begin{lm}\label{pat*}
Every symmetrized monomial $x_1x_2y^{(j_1,j_2)}$, where $(j_1,j_2)$ has the height $h=1$, is marked in 
$Div[x_1,x_2,y_1,y_2]$.
\end{lm}
\begin{proof}
The monomial $x_1x_2y^{(j_1,j_2)}$ is the leading term of the supersymmetric element $E_{2+j_1}$. 
\end{proof}

If $j_1=p-1$, then the claim of Lemma  \ref{pat*} is a particular case of Lemma \ref{sest*}.

\begin{lm}\label{sest*}
Every symmetrized monomial $x^{(i_1,i_2)}y^{(j_1,j_2)}$, where $r_1=r_2=1$, $k_2>0$ and $s_1=p-1$, is marked in 
$Div[x_1,x_2,y_1,y_2]$.
\end{lm}
\begin{proof}
We modify the proof of Lemma \ref{6.2}. To start, we recall the proof of Lemma \ref{6.2} restated for $Div[x_1,y_1,y_2]$.

If $s_2=p-1$, then $a_{i_1,j_1,j_2}$ does not appear in any defining equation and is therefore marked.
If $s_2\neq p-1$, then we have a series of $p-1-s_2$ pairs of equations
\[a_{s_2+1,j_s-s_s,j_1} + \ldots + (s_2+1) a_{1,j_2,j_1} + a_{0,j_2+1,j_1}=0\]
\[a_{p-1,j_1-p+1,j_2+1} + \ldots + (p-1) a_{1,j_1-1,j_2+1} + 
a_{0,j_1,j_2+1}=0,\]
\[a_{s_2+2,j_2-s_2,j_1-1} + \ldots + (s_2+2) a_{1,j_2+1,j_1-1} 
+ a_{0,j_2+2,j_1-1}=0\]
\[a_{p-2,j_1-p+1,j_2+2,j_1} + \ldots + (p-2) a_{1,j_1-2,j_2+2} 
+ a_{0,j_1-1,j_2+2}=0,\]
\[\ldots\]
\[\ldots\]
\[\begin{aligned}
&a_{p-1,j_2-s_2,j_1-p+s_2+2} + \ldots + (p-1) a_{1,j_2+p-s_2-2,j_1-p+s_2+2} \\
&+ a_{0,j_2+p-s_2-1,j_1-p+s_2+2}=0
\end{aligned}\]
\[\begin{aligned}
&a_{s_2+1,j_1-p+1,j_2+p-s_2-1} + \ldots + (s_2+1) a_{1,j_1-p+s_2+1,j_2+p-s_2-1} \\
&+ a_{0,j_1-p+s_2+2,j_2+p-s_2-1}=0,
\end{aligned}\]
that are all the defining equations involving variables 
\[a_{1,j_2,j_1}, a_{0,j_2+1, j_1},\]
\[a_{1,j_2+1,j_1-1}, \ldots, a_{1,j_2+p-s_2-1,j_1-p+s_2+1}\] and 
\[a_{0,j_1-1, j_2+2}, \ldots, a_{0,j_1-p+s_2+2,j_2+p-s_2-1}.\]

When we set all variables not listed above to zero and  $a_{1,j_1,j_2}=1$, the values of the remaining variables are
\[a_{0,j_2+1, j_1}=-(s_2+1),\]
\[a_{1,j_2+j,j_1-j}=(-1)^j\binom{s_2+j}{j}\]
for $j=1, \ldots, p-1-s_2$ and 
\[a_{0,j_1-j,j_2+j+1}=(-1)^{j+1}(s_2+j+1)\binom{s_2+j}{j}\] 
for $j=1, \ldots, p-2-s_2$.

Then
\[\begin{aligned}
&x_1 g_{j_1, j_2}-(s_2+1)g_{j_1,j_2+1}+
\sum_{j=1}^{p-1-s_2} (-1)^j\binom{s_2+j}{j} x_1 g_{j_2+j,j_1-j} \\
&+ \sum_{j=1}^{p-2-s_2} (-1)^{j+1}(s_2+j+1)\binom{s_2+j}{j} g_{j_1-j,j_2+j+1},
\end{aligned}\]
is a supersymmetric element that has the leading term $x_1y^{(j_1,j_2)}$, showing that $a_{1,j_1,j_2}$ is marked.

Now we make necessary adjustments for $Div[x_1,x_2,y_1,y_2]$.

If $s_2=p-1$, then $a_{i_1,i_2,j_1,j_2}$ does not appear in any defining equation and is therefore marked.
If $s_2\neq p-1$, then we have a series of $p-1-s_2$ quadruples of equations
\[a_{i_1+s_2,i_2,j_2-s_2,j_1} + \ldots + (s_2+1) a_{i_1,i_2,j_2,j_1} + a_{i_1-1,i_2,j_2+1,j_1}=0\]
\[a_{i_1,i_2+s_2,j_2-s_2,j_1} + \ldots + (s_2+1) a_{i_1,i_2,j_2,j_1} + a_{i_1,i_2-1,j_2+1,j_1}=0\]
\[a_{i_1+p-2,i_2,j_1-p+1,j_2+1} + \ldots + (p-1) a_{i_1,i_2,j_1-1,j_2+1} + a_{i_1-1,i_2,j_1,j_2+1}=0\]
\[a_{i_1,i_2+p-2,j_1-p+1,j_2+1} + \ldots + (p-1) a_{i_1,i_2,j_1-1,j_2+1} + a_{i_1,i_2-1,j_1,j_2+1}=0,\]
\[a_{i_1+s_2+1,i_2,j_2-s_2,j_1-1} + \ldots + (s_2+2) a_{i_1,i_2,j_2+1,j_1-1} + a_{i_1-1,i_2,j_2+2,j_1-1}=0\]
\[a_{i_1,i_2+s_2+1,j_2-s_2,j_1-1} + \ldots + (s_2+2) a_{i_1,i_2,j_2+1,j_1-1} + a_{i_1,i_2-1,j_2+2,j_1-1}=0\]
\[a_{i_1+p-3,i_2,j_1-p+1,j_2+2} + \ldots + (p-2) a_{i_1,i_2,j_1-2,j_2+2} + a_{i_1-1,i_2,j_1-1,j_2+2}=0\]
\[a_{i_1,i_2+p-3,j_1-p+1,j_2+2} + \ldots + (p-2) a_{i_1,i_2,j_1-2,j_2+2} + a_{i_1,i_2-1,j_1-1,j_2+2}=0,\]
\[\ldots\]
\[\ldots\]
\[\begin{aligned}
&a_{i_1+p-2,i_2,j_2-s_2,j_1-p+s_2+2} + \ldots + (p-1) a_{i_1,i_2,j_2+p-s_2-2,j_1-p+s_2+2} \\
&+ a_{i_1-1,i_2,j_2+p-s_2-1,j_1-p+s_2+2}=0
\end{aligned}\]
\[\begin{aligned}
&a_{i_1,i_2+p-2,j_2-s_2,j_1-p+s_2+2} + \ldots + (p-1) a_{i_1,i_2,j_2+p-s_2-2,j_1-p+s_2+2} \\
&+ a_{i_1,i_2-1,j_2+p-s_2-1,j_1-p+s_2+2}=0
\end{aligned}\]
\[\begin{aligned}
&a_{i_1+s_2,i_2,j_1-p+1,j_2+p-s_2-1} + \ldots + (s_2+1) a_{i_1,i_2,j_1-p+s_2+1,j_2+p-s_2-1} \\
&+ a_{i_1-1,i_2,j_1-p+s_2+2,j_2+p-s_2-1}=0
\end{aligned}\]
\[\begin{aligned}
&a_{i_1,i_2+s_2,j_1-p+1,j_2+p-s_2-1} + \ldots + (s_2+1) a_{i_1,i_2,j_1-p+s_2+1,j_2+p-s_2-1} \\
&+ a_{i_1,i_2-1,j_1-p+s_2+2,j_2+p-s_2-1}=0,
\end{aligned}\]
that are all the defining equations involving variables 
\[a_{i_1,i_2,j_2,j_1}, a_{i_1-1,i_2,j_2+1, j_1}, a_{i_1,i_2-1,j_2+1, j_1},\]
\[a_{i_1,i_2,j_2+1,j_1-1}, \ldots, a_{i_1,i_2,j_2+p-s_2-1,j_1-p+s_2+1},\] 
\[a_{i_1,i_2-1,j_1-1, j_2+2}, \ldots, a_{i_1,i_2-1,j_1-p+s_2+2,j_2+p-s_2-1}\] and 
\[a_{i_1-1,i_2,j_1-1, j_2+2}, \ldots, a_{i_1-1,i_2,j_1-p+s_2+2,j_2+p-s_2-1}.\]

When we set all variables not listed above to zero and $a_{i_1,i_2,j_1,j_2}=1$, then values of the remaining variables are
\[a_{i_1,i_2-1,j_2+1, j_1}=a_{i_1-1,i_2,j_2+1,j_1}=-(s_2+1),\]
\[a_{i_1,i_2,j_2+j,j_1-j}=(-1)^j\binom{s_2+j}{j}\]
for $j=1, \ldots, p-1-s_2$ and 
\[a_{i_1,i_2-1,j_1-j,j_2+j+1}=a_{i_1-1,i_2,j_1-j,j_2+j+1}=(-1)^{j+1}(s_2+j+1)\binom{s_2+j}{j}\] 
for $j=1, \ldots, p-2-s_2$.

Then
\[\begin{aligned}
&f_{i_1,i_2} g_{j_1, j_2}-(s_2+1)(f_{i_1-1,i_2}+f_{i_1,i_2-1})g_{j_1,j_2+1}+\sum_{j=1}^{p-1-s_2} (-1)^j\binom{s_2+j}{j} f_{i_1,i_2} g_{j_2+j,j_1-j} \\
&+ \sum_{j=1}^{p-2-s_2} (-1)^{j+1}(s_2+j+1)\binom{s_2+j}{j} (f_{i_1-1,i_2}+f_{i_1,i_2-1})g_{j_1-j,j_2+j+1},
\end{aligned}\]
is a supersymmetric element that has the leading term $x_1^{(i_1)}x_2^{(i_2)}y_1^{(j_1)}y_2^{(j_2)}$, showing that $a_{1,j_1,j_2}$ is marked.
\end{proof}

\begin{lm}\label{sedem*}
Every symmetrized monomial $x^{(i_1,i_2)}y^{(j_1,j_2)}$, where $r_1=r_2=2$, is marked in 
$Div[x_1,x_2,y_1,y_2]$.
\end{lm}
\begin{proof}
We modify the proof of Lemma \ref{6.3}.

The statement is true if $(j_1,j_2)$ has the height $h=1$ by Lemma \ref{sest*}. 

Consider the following defining equations

\begin{equation}\tag{$A_{1t}$}
\begin{aligned}
&a_{i_1+s_1-t,i_2,j_1-s_1,j_2+t} + \ldots +\binom{s_1+2-t}{2}a_{i_1,i_2,j_1-t,j_2+t}\\
&+(s_1+2-t)a_{i_1-1,i_2,j_1+1-t,j_2+t}+a_{i_1-2,i_2,j_1+2-t,j_2+t}=0
\end{aligned}
\end{equation}
for $t=0, 1$;

\begin{equation}\tag{$A_{2t}$}
\begin{aligned}
&a_{i_1,i_2+s_1-t,j_1-s_1,j_2+t} + \ldots +\binom{s_1+2-t}{2}a_{i_1,i_2,j_1-t,j_2+t}\\
&+(s_1+2-t)a_{i_1,i_2-1,j_1+1-t,j_2+t}+a_{i_1,i_2-2,j_1+2-t,j_2+t}=0
\end{aligned}
\end{equation}
for $t=0,1$;

\begin{equation}\tag{$B_{1t}$}
\begin{aligned}
&a_{i_1+s_2-t,i_2,j_1+t,j_2-s_2} + \ldots +\binom{s_2+2-t}{2}a_{i_1,i_2,j_1+t,j_2-t}\\
&+(s_2+2-t)a_{i_1-1,i_2,j_1+t,j_2+1-t}+a_{i_1-2,i_2,j_1+t,j_2+2-t}=0
\end{aligned}
\end{equation}
for $t=0,1$;

\begin{equation}\tag{$B_{t2}$}
\begin{aligned}
&a_{i_1,i_2+s_2-t,j_1+t,j_2-s_2} + \ldots +\binom{s_2+2-t}{2}a_{i_1,i_2,j_1+t,j_2-t}\\
&+(s_2+2-t)a_{i_1,i_2-1,j_1+t,j_2+1-t}+a_{i_1,i_2-2,j_1+t,j_2+2-t}=0
\end{aligned}
\end{equation}
for $t=0,1$;

If $j_1>j_2$, then 
\[\begin{aligned}&f_{i_1,i_2}g_{j_1,j_2} - \frac{s_1+1}{2} f_{i_1-1,i_2}g_{j_1+1,j_2} - 
\frac{s_1+1}{2} f_{i_1,i_2-1}g_{j_1+1,j_2} \\
&- \frac{s_2+1}{2} f_{i_1-1,i_2}g_{j_1+1,j_2} -\frac{s_2+1}{2} f_{i_1,i_2-1}g_{j_1,j_2+1} \\
&+\frac{(s_1+1)(s_2+1)}{2} f_{i_1-1,i_2-1}g_{j_1+1,j_2+1}
\end{aligned}\]
is a supersymmetric element that has the leading term $x^{(i_1,i_2)}y^{(j_1,j_2)}$, showing that $a_{i_1i_2j_1j_2}$ is marked.
This is easy to verify since the variables corresponding to summands in this expression appear only in the 
defining equations $(A_{1t}), (A_{2t}), (B_{1t}), (B_{2t})$ for $t=0,1$.

If $j_1=j_2=j$, then $s_1=s_2=s$ and
\[f_{i_1,i_2}g_{j,j} - \binom{s+2}{2}f_{i_1-1,i_2-1}g_{j+2,j}\] 
is a supersymmetric element that has the leading term $x^{(i_1,i_2)}y^{(j,j)}$ showing that $a_{i_1i_2jj}$ is marked. 
This is easy to verify this since the variables corresponding to summands in this expression appear only in the 
defining equations $(A_{0t})=(B_{0t})$ and $(A_{1t})=(B_{1t})$. 
\end{proof}

\begin{rem}
The modification of the proofs of Lemmas \ref{sest*} and \ref{sedem*} from the proofs of Lemmas \ref{6.2} and \ref{6.3} by "doubling" the value of $r$ to 
$r_1=r_2$ shows how to proceed in general to build marked elements in $Div[x_1, \ldots, x_m,y_1,y_2]$ from those in $Div[x_1, y_1,y_2]$ in the case when $r_1=\ldots=r_m$.
It is not clear if the assumption $r_1=\ldots=r_m$ is really needed, but we require the modification only in those cases.
\end{rem}

\begin{pr}\label{styri}
The symmetrized monomials $x^{(i_1,i_2)}y^{(j_1,j_2)}$ such that 
\begin{itemize}
\item $p$ divides $i_1$ or $i_2$ and either $i_2>0$ or $j_1>0$; or
\item $r_1=1$ or $r_2=1$ and $(j_1,j_2)$ has height $h>1$; or 
\item $r_1>0$, $r_2=1$, $k_2>0$, $s_1\neq p-1$ and $(j_1,j_2)$ has height $h=1$
\end{itemize}
are unmarked in $Div[x_1,x_2,y_1,y_2]$.
\end{pr}
\begin{proof}
The first part of the statement follows from Lemma \ref{11} and Corollary \ref{cor2}.
The second part follows from Lemmas \ref{6.5} and \ref{connect}.

For the last part, take $(j_1,j_2)$ of the form $(j_1,0)$, where $0\leq j_1<p-1$ and consider the defining equation

\[a_{i_1,i_2+s_1,j_1-s_1,0}+ \ldots + (s_1+1)a_{i_1,i_2,j_1,0}+a_{i_1,i_2-1,j_1+1,0}=0.\]
By Lemma \ref{blow}, we have $a_{i_1,i_2-1,j_1+1,0}=a_{i_1+i_2-1,0,j_1+1,0}$. Since $a_{i_1+i_2-1,0,j_1+1,0}$ and all variables
$a_{i_1,i_2+s_1-i,j_1-s_1+i,0}$ for $i=0, \ldots, s_1-1$ are higher than $a_{i_1,i_2,j_1,0}$ we conclude that $a_{i_1,i_2,j_1,0}$ is unmarked.
\end{proof}

\begin{theorem}
The algebra $\mathcal{S}$ of all supersymmetric elements in $Div[x_1,x_2,y_1,y_2]$ is generated by elements $S(M)$, where $M$ is one of the symmetrized monomials 
\begin{itemize}
\item{$x_1$};
\item{$x_1^{(p^s)}$ for some $s>0$};
\item{$x_1x_2y^{(j_1,j_2)}$, where $(j_1,j_2)$ has height $h=1$;}
\item{$x^{(i_1,i_2)}y^{(j_1,j_2)}$, where $r_1=r_2=1$, $k_2>0$ and $s_1=p-1$;}
\item{$x^{(i_1,i_2)}y^{(j_1,j_2)}$, where $r_1=r_2=2$.}
\end{itemize}
\end{theorem}
\begin{proof}
Lemmas \ref{span3}, \ref{p}, \ref{pat*}, \ref{sest*} and \ref{sedem*} describe all marked monomials.
Proposition \ref{styri} describes all unmarked monomials. Lemma \ref{span3} and Proposition \ref{generators} conclude the proof.
\end{proof}

\end{document}